\documentclass[letterpaper,10pt,conference,romanappendices]{ieeeconf}

\IEEEoverridecommandlockouts
\usepackage{amsmath,amssymb}
\usepackage{graphicx,color}
\usepackage{palatino}
\usepackage{algorithmic,algorithm}
\usepackage{psfrag}
\usepackage{dsfont}
\usepackage{enumerate}

\newcommand{\real}{{\mathbb{R}}}
\newcommand{\realpositive}{{\mathbb{R}}_{>0}}

\newcommand{\rank}{\operatorname{rank}}
\newcommand{\nullity}{\operatorname{nullity}}
\newcommand{\spec}{\operatorname{spec}} % spectrum of a matrix
\newcommand{\until}[1]{\{1,\dots,#1\}}

\newcommand{\subscr}[2]{#1_{\textup{#2}}}
\newcommand{\supscr}[2]{#1^{\textup{#2}}}
\newtheorem{theorem}{Theorem}[section]

\newtheorem{remark}[theorem]{Remark}

\newtheorem{proposition}[theorem]{Proposition}
\newtheorem{definition}[theorem]{Definition}
\newtheorem{claim}[theorem]{Claim}

\newcommand\oprocendsymbol{\hbox{$\bullet$}}
\newcommand\oprocend{\relax\ifmmode\else\unskip\hfill\fi\oprocendsymbol}

\begin{document}

\title{On Single-Input Controllable Linear Systems Under Periodic DoS
  Jamming Attacks}

\author{Hamed Shisheh Foroush and Sonia Mart{\'\i}nez \thanks{The
    authors are with Department of Mechanical and Aerospace
    Engineering, University of California, San Diego, 9500 Gilman Dr,
    La Jolla CA, 92093, {\tt\small hshisheh,soniamd@ucsd.edu}} }%
\maketitle

\begin{abstract}
  In this paper, we study remotely-controlled single-input
  controllable linear systems, subject to periodic Denial-of-Service
  (DoS) attacks.  We propose a control strategy which can beat any
  partially identified jammer by properly placing
  the closed-loop poles. This is proven theoretically for systems of
  dimension $4$ or less. Nevertheless, simulations show the
  practicality of this strategy for systems up to order
  $5$. 
\end{abstract}

\section{Introduction}

\textit{Cyber-physical systems} consist of physical networked systems
which are controlled and monitored remotely~\cite{AC-SA-SSS:08}.
Novel advances in communications and sensing technologies have
promoted the emergence of these systems, which bears numerous
advantages ranging from ease of implementation to increased utility in
infrastructure facilities~\cite{JH-PN-YX:07}. However, the potential
benefits of such systems may be overturned by several challenges that
include a higher exposure to external attacks. This has motivated a
renewed research effort in the area of \textit{system
  security}~\cite{AC-SA-BS-AG-AP-SS:09,NA:10}, which attempts to
address system preservation issues.

In particular, the security of cyber-physical systems can especially
be threatened by communication-signal jammers that are exogenous to
the system. Common attacks include \textit{Denial-of-Service (DoS)}
and \textit{Deceptive} attacks. 
   In brief, a Deceptive attacker aims to change parceled data,
whereas a DoS one tries to corrupt the transmitted
data~\cite{WX-WT-YZ-TW:05,RP:04}. According to~\cite{EB-JL:04}, the
most likely type of attack is a DoS attack. These attacks can be
further categorized into \emph{periodic} or \textit{Pulse-Width
  Modulated (PWM)} jammers motivated by ease of implementation,
detection avoidance, and energy constraints while; see the
papers~\cite{XL-EWWC-RKCC:09,XL-RKCC:05,BDB-PT:11,AGF-VAS-NP:10}.
Inspired by this facts, this work focuses on the compensation of PWM
DoS jamming attacks whose periodicity has already been detected.

The subject of security of cyber-physical systems is receiving wide
attention in the controls community. In the context of multiagent
systems, the works~\cite{SS-CNH:08ab,FP-AB-FB:11,FP-RC-FB:11}, aim to identify
malicious agents who are part of the network in order to cancel their
contribution. The main goal of~\cite{SB-TB:10} and~\cite{SB-TB:10b} is
to maintain group connectivity despite the presence of a malicious
agent, thus identification is off the ground. The
paper~\cite{MZ-SM:12-auto} proposes a Receding Horizon Control
methodology to deal with a class of deceptive \textit{Replay
  attackers} inducing system delays in formation control missions.
Our problem setup is related to these studies in the sense that the
jammer is assumed to be detected and we aim to develop a method to
overcome its effect.

In the framework of secure discrete LTI systems,~\cite{HF-PT-SD:11}
considers deceptive attacks where the observation channel is jammed.
In~\cite{SA-AC-SSS:09}, however, a DoS attack where the attacker
corrupts the channel while obeying an Identically Independent
Distributed (IID) assumption is considered. Similar research is
conducted in~\cite{LS-BS-MF-KP-SSS:07}.

Game Theory is a natural framework to study system security; some
representative references include~\cite{AG-CL-TB:10,GT-JSB:08},
and~\cite{SR-CE-SS-DD-VS-QW:10}. These papers model the security
problem as a (dynamic) zero-sum, non-cooperative game in order to
predict the behavior of the attacker.  Inspired by a leader-follower
game-theoretic formulation, the paper~\cite{MZ-SM:11b} employs
reinforcement learning to beat a deceptive attacker that can be
modeled as a linear map.  In this framework, the closest reference to
our research is~\cite{AG-CL-TB:10}, which studies how the optimal
control of linear system remotely under a strategic,
energy-constrained DoS jammer. A main restriction
of~\cite{AG-CL-TB:10} is the consideration of scalar dynamics, which
makes the presented analysis more tractable.

Motivated by the goal of maintaining ``intelligent'' and economic
communications, here we address the problem of system resilience in
the context of \textit{triggering control}, i.e., control actions
triggered only when it is necessary.  This is inspired by recent
work~\cite{PT:07,MM-AA-PT:10,XW-MDL:09}.  The distinctive feature in
our study is the fact that communication is not always feasible.  

We consider partially identified DoS attacks imposed by PWM jammers,
along with single-input LTI systems. The current work follows
upon~\cite{HSF-SM:12-cdc}, where we provide sufficient conditions on
the jammer's parameters which, in conjunction with a given triggering
law, can ensure system stability.  Despite the fact that it covers a
broad class of continuous LTI systems, the previous strategy cannot
cope with every given DoS periodic jammer. The current manuscript
solves this issue by introducing a parameter-dependent control
strategy which can tackle any DoS periodic jammer. Thus it ensures
asymptotic stability under this class of attacks, i.e., it renders the
the system \textit{safe and secure}. The strategy is proven to work
for a narrower class of continuous LTI systems of order $4$, or less,
and, based on simulations, we conjecture its validity for a large
class of higher-order systems as well.

 The rest of the paper is organized as
 follows. Section~\ref{sec-problem-formulation} includes the problem
 formulation and notations. In Section~\ref{sec-prel}, some
 preliminaries are provided, where we propose a control design law to
 be employed later. Then we go on to
 Section~\ref{sec-jordan-decomp-triggering}, where we discuss a novel
 attack-resilient triggering law consistent with the jammer signal. In
 Section~\ref{sec-stab-analysis}, we analyze and prove the security of
 the system equipped with this control design and triggering law.  In
 Section~\ref{sec-simulations}, we demonstrate the functionality of
 our theoretical results on two academic examples. At last, in
 Section~\ref{sec-conclusions-future-work} we summarize the results
 and state the future work.

%%%%%%%%%%%%%%%%%%%%%%%%%%%%%%

\section{Problem Formulation}
\label{sec-problem-formulation}

In this section, we state, both formally and informally, the main
problem analyzed in the paper. 

We consider a remote operator-plant setup, where the operator uses a
control channel to send wirelessly a control command to an unstable
plant, see Figure~\ref{figure-prob-arch}. We assume that the plant has
no specific intelligence and is only capable of updating the control
based on the data it receives. We also assume that the operator knows
the plant dynamics and is able to measure its states
at particular time-instants.\footnote{This information can be obtained by using
  either local ``passive'' sensors, e.g., camera network or
  positioning systems, e.g., GPS, where no communication or cheap and
  safe communication is required.} 

\begin{figure}
\centering
\psfrag{Plant}[tc][cc]{Plant} \psfrag{Operator}[tc][cc]{Operator}
\psfrag{Jammer}{\textcolor{red}{Jammer}} \psfrag{Control}{Control}
\psfrag{Channel}{Channel}
\psfrag{Observation}{Observation}
\psfrag{Channel}{Channel}
\psfrag{(secure)}{(secure)}
\includegraphics[scale=0.6]{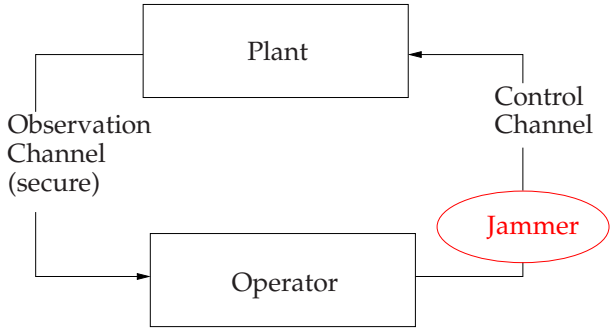}
\caption{Problem Architecture}
\label{figure-prob-arch}
\end{figure}

More precisely, consider the following closed-loop dynamics:
\begin{subequations}
\begin{align}
\label{eqn-sys-high-dim-intro}
\dot{x}(t)&=Ax(t)+Bu(t)\,,\\
\label{eqn-sys-high-dim-control-law-intro}
u(t)&=Kx(t_{k})\,,\quad \forall t\in[t_{k},t_{k+1}[\,,
\end{align}
\end{subequations}
where $x\in\real^{n}$ is the state vector, $u\in\real$ is the
input, $A$, $B$ and $K$ are matrices of proper dimensions, and
$\left\{t_{k}\right\}_{k\geq 1}$ is a triggering time sequence. Here,
we also assume that: (i) the system~\eqref{eqn-sys-high-dim-intro} is
open-loop unstable, and (ii) the pair $(A,B)$ is controllable.

We consider an \textit{energy-constrained}, \textit{periodic} jammer
whose signal can be represented as follows:
\begin{equation}
\label{eqn-jammer-signal}
\subscr{u}{jmd}(t)=\begin{cases}
  1, & (n-1)T\leq t\leq (n-1)T+\subscr{T}{off},\\
  0, & (n-1)T+\subscr{T}{off}\leq t\leq nT,
\end{cases}
\end{equation}
where $n\in\mathbb{N}$ is the period number, $T\in\realpositive$, and
$\mathcal{T}=[0, T]$ is the action-period of the jammer. Also,
$\subscr{T}{off}\in\realpositive$, $\subscr{T}{off} <T$, and
$\mathcal{\subscr{T}{off}}=[0, \subscr{T}{off}]$ is the time-period
where it is sleeping, so communication is possible.  We further denote
$\subscr{T}{on}\in\realpositive$, and
$\mathcal{\subscr{T}{on}}=[\subscr{T}{on}, T]$ to be the time-period
where the jammer is active, thus no data can be sent. Accordingly, it
holds that $\subscr{T}{off} + \subscr{T}{on} = T$. We also note that
the parameter $\subscr{T}{off}$ need not be time-invariant
which recalls Pulse-Width Modulated (PWM) jamming.  Finally, we denote
by $\supscr{\subscr{T}{off}}{cr}$ a uniform lower-bound for
$\subscr{T}{off}$, i.e., $\supscr{\subscr{T}{off}}{cr} \leq
\subscr{T}{off}$ which we assume holds for all the periods and we have
identified as well.

In this paper, we assume that the type of jammer and the period of the
jamming signal has been identified. Future work will be devoted to
enlarge the triggering time sequence for identification
purposes.

Putting these pieces together, we study the following problem:
\begin{quote}
  \emph{[Problem formulation]}: Consider any energy-constrained,
  periodic jammer described by~\eqref{eqn-jammer-signal} with
  parameters $T$ and $\subscr{T}{off}$. Knowing $T$ and
  $\supscr{\subscr{T}{off}}{cr}$, a uniform lower bound on the
  jammer's sleeping periods, find a control strategy of the
  form~\eqref{eqn-sys-high-dim-control-law-intro} that is resilient to
  the action of this jammer.
\end{quote}

\section{Preliminaries}
\label{sec-prel}

In this section, we recall some useful properties of the systems that
we study. These will be employed in the subsequent analysis.

Since $(A,B)$ is a controllable pair, the
system~\eqref{eqn-sys-high-dim-intro} can be put into a controllable
canonical form by a proper similarity transformation~\cite{DSB:05}. Based
on this fact, we narrow our study down to the systems of this form:
\begin{align}
\label{eqn-sys-matrix-prel}
\dot{x} &= \left[ \begin{array}{ccccc} 0 & 1 & 0 & \cdots & 0 \\
 0 & 0 & 1 & \cdots & 0 \\ 
\vdots & \vdots & \vdots & \ddots & \vdots \\ 
0 & 0 & 0 & \cdots & 1 \\ 
-a_{n} & -a_{n-1} & -a_{n-2} & \cdots & -a_{1} \end{array} \right] x + 
\left[ \begin{array}{c} 0 \\ 0 \\ \vdots \\ 0 \\ 1 \end{array} \right] u\,, \nonumber \\
u &= \left[-k_{n}+a_{n}, -k_{n-1}+a_{n-1}, \cdots, -k_{1}+a_{1}\right]x\,.
\end{align}
Recalling the pole-placement assignment techniques, we obtain the
folowing result:
\begin{proposition}
\label{prop-lambda-eigenvalue-prel}
Consider $\lambda\in \realpositive$ and
system~\eqref{eqn-sys-matrix-prel}.  By choosing: 
\begin{equation*}
  k_{i} = \left( \begin{array}{c} n \\ i \end{array} \right) \lambda^{i}, \quad i \in \until{n}\,,
\end{equation*}
all the closed-loop system poles are placed at  $-\lambda $.
\end{proposition}
\begin{proof}
  By plugging the proposed gains, $K_\lambda = [k_1,\dots,k_n]$, in the
  dynamics~\eqref{eqn-sys-matrix-prel}, the state matrix of the
  closed-loop system becomes:
\begin{align*}
  A&+BK_{\lambda} =\\
  & \left[ \begin{array}{cccc} 0 & 1 & \cdots & 0 \\
      0 & 0 & \cdots & 0 \\
      \vdots & \vdots & \ddots & \vdots \\
      0 & 0 & \cdots & 1 \\
      -\left( \begin{array}{c} n \\ n \end{array} \right) \lambda^{n}
      & -\left( \begin{array}{c} n \\ n-1 \end{array} \right)
      \lambda^{n-1} & \cdots & -n\lambda \end{array} \right]\,.
\end{align*}
Note that the characteristic polynomial of this matrix, thanks to its
specific structure, is given by~\cite{DSB:05}:
\begin{equation*} 
  p(s)=s^n + n \lambda s^{n-1}+ \cdots +\left( 
    \begin{array}{c} 
      n \\ n-1 
    \end{array} \right) \lambda^{n-1} s + \left( 
    \begin{array}{c} 
      n \\ n 
    \end{array} \right) 
  \lambda^{n}\,.
\end{equation*} 
Observe that the latter polynomial is indeed $p(s)=(s +
\lambda)^n$. Recall also that the eigenvalues of the matrix
$A+BK_{\lambda}$, i.e., the poles of the closed-loop system, are the
roots of $p(s)$. Thus, we conclude that with this choice of gains, all
the closed-loop poles are placed at $-\lambda$.
\end{proof}

The multiplicities of the eigenvalue $-\lambda$, described in the
previous result, are further characterized next:

\begin{proposition}
\label{prop-lambda-eigenvalue-prop-prel}
Consider the system~\eqref{eqn-sys-matrix-prel} along with the gains
given in Proposition~\ref{prop-lambda-eigenvalue-prel}. The eigenvalue
$-\lambda$ of the matrix $A+BK_{\lambda}$ has algebraic multiplicity
$n$ and geometric multiplicity $1$.  

\end{proposition}
\begin{proof}
  Since the characteristic polynomial of $A+BK_{\lambda}$ is $p(s)=(s
  + \lambda)^n$, the algebraic multiplicity of $-\lambda$ is equal to
  $n$. The geometric multiplicity of $-\lambda$ is equal to the number
  of linearly independent eigenvectors of $A+BK_{\lambda}$ associated
  to $-\lambda$; in other words, the nullity of the matrix
  $A+BK_{\lambda}+\lambda I$. Further, it holds that~~\cite{DSB:05}:
\begin{equation}
\label{eqn-dummy-null-rank}
\nullity(A+BK_{\lambda}+\lambda I) = n - \rank(A+BK_{\lambda}+\lambda I)\,.
\end{equation}
By construction, $\rank(A+BK_{\lambda}+\lambda I) \geq n-1$, as it has
$n-1$ linearly independent columns. Also, since $-\lambda$ is an
eigenvalue of $A+BK_{\lambda}$, then it holds that $\det
(A+BK_{\lambda}+\lambda I) = 0$, thus $\rank(A+BK_{\lambda}+\lambda I)
< n$, From here, $\rank(A+BK_{\lambda}+\lambda I)=n-1$. According
to~\eqref{eqn-dummy-null-rank}, we have that
$\nullity(A+BK_{\lambda}+\lambda I) = 1$, hence the geometric
multiplicity of $-\lambda$ is $1$.
\end{proof}

\begin{remark}\label{remark-lambda-eigenvalue-prel-poly-depend}
  Note that the matrix $A+BK_{\lambda}$ has only one linearly
  independent eigenvector, therefore, it is \textit{not}
  diagonalizable. This property holds for all values of $\lambda \in
  \realpositive$.

  Furthermore, since the matrix $A+BK_{\lambda}+\lambda I$ depends on
  $\lambda$ in a polynomial way, the components of this eigenvector
  $v$ are rational functions of $\lambda$. In fact, $v$ can be found
  as the solution to the following equation:
\begin{equation*}
  (A+BK_{\lambda}+\lambda I)v = 0\,.
\end{equation*}
\end{remark}

%%%%%%%%%%%%%%%%%%%%%%%%%%%%
\section{Jordan Decomposition and Triggering Strategy}
\label{sec-jordan-decomp-triggering}

Our control strategy will consist of choosing an appropriate
$K_\lambda$ as in Proposition~\ref{prop-lambda-eigenvalue-prel} and an
associated triggering strategy $\{t_k\}_{k \ge 1}$. In this section,
we study the Jordan decomposition of the closed-loop system under
$K_\lambda$, and how to choose the corresponding $\{t_k\}_{k \ge 1}$.

From Proposition~\ref{prop-lambda-eigenvalue-prel}, the eigenvalues of
the matrix $A+BK_{\lambda}$ are at $-\lambda$. Thus, the Jordan
decomposition of this matrix can be expressed as:
\begin{equation}
  \label{eqn-jordan-decomp-closed-loop-matrix}
  A+BK_{\lambda} = T_{\lambda} J_{\lambda} T^{-1}_{\lambda}\,,
\end{equation}
where $J_{\lambda}=-\lambda I+N$ and $T_{\lambda}$ is a matrix built
upon the linearly independent and  generalized eigenvectors. We
would like to remark the following:
\begin{itemize}

\item The matrix $N$ has a \textit{unique} structure for all values of
  $\lambda$; this is because the geometric multiplicity of this
  eigenvalue remains unchanged. Moreover, by construction of Jordan
  decomposition, it does not depend on this parameter,

\item As discussed in the
  Remark~\ref{remark-lambda-eigenvalue-prel-poly-depend}, the only
  linearly independent eigenvector of $A+BK_{\lambda}$ depends in a
  rational way on $\lambda$. Then, by construction of the generalized
  eigenvectors~\cite{VNF:58}, they also rationally depend on
  $\lambda$. Hence, the matrices $T_{\lambda}$ and $T_{\lambda}^{-1}$,
  they also depend on $\lambda$ in a rational way.

\end{itemize}

Before presenting our triggering strategy, we introduce a family of
coordinate transformations used in this paper. They are based on the
Jordan decomposition of previous paragraphs. Let us consider the
system~\eqref{eqn-sys-matrix-prel}, with the control
$u(t)=K_{\lambda}x(t_{k})$. Then, the closed-loop dynamics is:
\begin{equation*}
  \dot{x} = (A+BK_{\lambda})x + BK_{\lambda}e\,,
\end{equation*}
where $e(t) = x(t_{k}) - x(t)$. Recalling~\eqref{eqn-jordan-decomp-closed-loop-matrix}, the latter
dynamics under the static transformations
$e(t)=T_{\lambda}e_{\lambda}(t)$, and $x(t) =
T_{\lambda}x_{\lambda}(t)$, yields:
\begin{equation}
  \label{eqn-jordan-triggering-trans-sys}
  \dot{x}_{\lambda} =
  J_{\lambda}x_{\lambda}+T_{\lambda}^{-1}BK_{\lambda}T_{\lambda}e_{\lambda}\,.
\end{equation}
The following result states our first attempt in developing the
triggering strategy.
\begin{proposition}
  \label{prop-jordan-trig-seq-event-trig-cond}
  Take $\lambda> \|N\|+1/2$ and $K_{\lambda}$ as in
  Proposition~\ref{prop-lambda-eigenvalue-prel}. Then
  $V(x_{\lambda})=x_{\lambda}^{T}x_{\lambda}$ is an ISS-Lyapunov
  function for the system~\eqref{eqn-jordan-triggering-trans-sys} and
  the event-triggering condition:
  \begin{equation}
    \label{eqn-prop-event-trig-cond}
    |e_{\lambda}(t)|^{2} \leq \frac{\sigma (2\lambda -1
      -2\|N\|)}{\|T_{\lambda}^{-1}BK_{\lambda}T_{\lambda}\|^{2}}
    |x_{\lambda}(t)|^{2}\,,
  \end{equation}
  guarantees the asymptotic stability of the system, for
  $\sigma\in(0,1)$.
\end{proposition}
\begin{proof}
Let $\bar{B}_{\lambda} \triangleq
T_{\lambda}^{-1}BK_{\lambda}T_{\lambda}$, computing the
time-derivative of $V(x_{\lambda})$, and plugging from
dynamics~\eqref{eqn-jordan-triggering-trans-sys}, we obtain:
\begin{align*}
  \dot{V}=\dot{x}_{\lambda}^{T}x_{\lambda} +
  x_{\lambda}^{T}\dot{x}_{\lambda}=&(J_{\lambda}x_{\lambda}+\bar{B}_{\lambda}e_{\lambda})^{T}x_{\lambda}
  +\\ &
  x_{\lambda}^{T}(J_{\lambda}x_{\lambda}+\bar{B}_{\lambda}e_{\lambda})\,.
\end{align*}
After some simplification, we obtain:
\begin{equation*}
  \dot{V}=x_{\lambda}^{T}(J_{\lambda}^{T}+J_{\lambda})x_{\lambda}+
  e_{\lambda}^{T}\bar{B}_{\lambda}^{T}x_{\lambda}+x_{\lambda}^{T}\bar{B}_{\lambda}e_{\lambda}\,.
\end{equation*}
The latter can be further bounded recalling the following inequality:
\begin{equation*}
  e_{\lambda}^{T}\bar{B}_{\lambda}^{T}x_{\lambda}+
  x_{\lambda}^{T}\bar{B}_{\lambda}e_{\lambda} \leq
  x_{\lambda}^{T}x_{\lambda} +
  e_{\lambda}^{T}\bar{B}_{\lambda}^{T}\bar{B}_{\lambda}e_{\lambda}\,,
\end{equation*}
which then yields:
\begin{equation}
  \label{eqn-prop-event-trig-dummy-1}
  \dot{V} \leq
  x_{\lambda}^{T}(J_{\lambda}^{T}+J_{\lambda}+I)x_{\lambda}+
  e_{\lambda}^{T}\bar{B}_{\lambda}^{T}\bar{B}_{\lambda}e_{\lambda}\,.
\end{equation}
Now, from our discussion on Jordan decomposition, it holds that
$J_{\lambda}^{T}+J_{\lambda} = -2\lambda I +N^{T} + N$.  By plugging
this back into~\eqref{eqn-prop-event-trig-dummy-1}, we obtain:
\begin{equation*}
  \dot{V} \leq x_{\lambda}^{T}(N^{T}+N-(2\lambda-1)
  I)x_{\lambda}+e_{\lambda}^{T}\bar{B}_{\lambda}^{T}\bar{B}_{\lambda}e_{\lambda}\,.
\end{equation*}
We further upper-bound the latter equation, noting that
$\|N^{T}\|=\|N\|$:
\begin{equation}
  \label{eqn-prop-event-trig-dummy-2}
  \dot{V} \leq -(2\lambda -1 -2\|N\|)|x_{\lambda}|^{2} +
  \|\bar{B}_{\lambda}\|^{2}|e_{\lambda}|^{2}\,.
\end{equation}
Hence, for $\lambda > \|N\|+1/2$, according to
Equation~\eqref{eqn-prop-event-trig-dummy-2}, and that $V(x_{\lambda})
= |x_{\lambda}|^{2}, \forall x_{\lambda}$, we conclude that
$V(x_{\lambda}) = x_{\lambda}^{T}x_{\lambda}$ is an ISS-Lyapunov
function for the
system~\eqref{eqn-jordan-triggering-trans-sys}. Moreover, let
$\sigma\in(0,1)$ and that the triggering time-sequence be given when
the following condition is violated:
\begin{equation*}
  |e_{\lambda}|^{2} \leq \frac{\sigma
    (2\lambda-1-2\|N\|)}{\|\bar{B}_{\lambda}\|^{2}}
  |x_{\lambda}|^{2}\,,
\end{equation*}
it then holds that:
\begin{equation*}
  \dot{V} \leq -(1-\sigma) (2\lambda-1-2\|N\|)|x_{\lambda}|^{2}\,.
\end{equation*} 
Hence, the event-triggering condition, described
by~\eqref{eqn-prop-event-trig-cond}, guarantees the asymptotic
stability of the system.
\end{proof}

\begin{remark} 
  \label{remark-no-accum-event-trig}
  Let $t_{k}$ and $t_{k+1}$ be two consecutive time-instants given by
  the triggering law~\eqref{eqn-prop-event-trig-cond}. Then, for each
  $\lambda$, the following holds:
  \begin{equation*}
    \exists \tau_{\lambda}>0,\, \text{such that}\, t_{k+1}-t_{k}\geq
    \tau_{\lambda}, \forall k\in\mathbb{N}\,.
  \end{equation*}
  This is based on Theorem~III.1,
    presented
  in~\cite{PT:07}. In other words, the time-sequence generated by the
  triggering law~\eqref{eqn-prop-event-trig-cond} does not
  accumulate. This is an important observation used in our analysis.
\end{remark}
For the parameter $\tau_{\lambda}$, we show the following property:
\begin{theorem}
  \label{thm-jordan-trig-seq-tau-lambda}
  Consider the parameter $\tau_{\lambda}$ introduced in
  Remark~\ref{remark-no-accum-event-trig}. Then, the following holds:
  \begin{equation}
    \label{eqn-prop-tau-lambda-limit}
    \lim_{\lambda \rightarrow \infty} {\tau_{\lambda}} = 0\,.
  \end{equation}
\end{theorem} 
\begin{proof}
  Recalling
Remark~\ref{remark-no-accum-event-trig}, we first note that for the
parameter $\tau_{\lambda}$ it holds that $\tau_{\lambda} \leq t_{k+1}
- t_{k}, \forall k\in \mathbb{N}$. In particular, it holds that
$\tau_{\lambda} \leq t_{2} - t_{1}$.  Let us, without loss of
generality, set $t_{1} = 0$, and denote $t_{\lambda} \triangleq
t_{2}$. Then, it holds that $0 \leq \tau_{\lambda} \leq
t_{\lambda}$. In this proof, we shall show $\lim_{\lambda \rightarrow
  \infty} {t_{\lambda}} = 0$, which implies that $\lim_{\lambda
  \rightarrow \infty} {\tau_{\lambda}} = 0$. By construction of the
triggering law~\eqref {eqn-prop-event-trig-cond}, the time-instant
$t_{\lambda}$ is when the following holds:
 \begin{equation}
   \label{eqn-thm-tau-event}
   |e_{\lambda}(t_{\lambda})| = \frac{\sqrt{\sigma (2\lambda -1
       -2\|N\|)}}{\|T_{\lambda}^{-1}BK_{\lambda}T_{\lambda}\|}
   |x_{\lambda}(t_{\lambda})|\,.
 \end{equation}
In the latter inequality and according
to~\eqref{eqn-jordan-decomp-closed-loop-matrix}, it holds that
$BK_{\lambda} = T_{\lambda} J_{\lambda} T_{\lambda}^{-1} - A$, that
is:
\begin{equation*} 
  T_{\lambda}^{-1} BK_{\lambda} T_{\lambda} = J_{\lambda} -
  T_{\lambda}^{-1} A T_{\lambda}\,.
\end{equation*}
Hence, according to this equation,~\eqref{eqn-thm-tau-event} can be
written as follows:
\begin{equation*}
  |e_{\lambda}(t_{\lambda})| = \frac{\sqrt{\sigma (2\lambda -1
      -2\|N\|)}}{\|J_{\lambda} - T_{\lambda}^{-1}AT_{\lambda}\|}
  |x_{\lambda}(t_{\lambda})|\,.
\end{equation*}
Let us denote:
\begin{equation}
  \label{eqn-thm-tau-f-lambda-notion}
  F(\lambda) \triangleq \frac{\sqrt{\sigma (2\lambda -1
      -2\|N\|)}}{\|J_{\lambda} - T_{\lambda}^{-1}AT_{\lambda}\|}\,.
\end{equation}
We continue by presenting the following result:

%%%%%% claim %%%%%%
\begin{claim}\label{claim-in-thm-tau-f-lambda}
  The following holds:
\begin{equation}
  \label{eqn-claim-thm-tau-f-lambda-limit}
  \lim_{\lambda \rightarrow \infty}{F(\lambda)} = 0\,.
\end{equation}
\end{claim}
\emph{Proof of Claim~\ref{claim-in-thm-tau-f-lambda}:} Recalling
that $J_{\lambda} = -\lambda I + N$,
Equation~\eqref{eqn-thm-tau-f-lambda-notion} can be written as:
\begin{equation*}
  F(\lambda) = \frac{\sqrt{\sigma (2\lambda -1 -2\|N\|)}}{\|(-\lambda
    I - T_{\lambda}^{-1}AT_{\lambda}) -(-N)\|}\,,
\end{equation*} 
where by applying the triangular-inequality of the form $|\|-\lambda I
- T_{\lambda}^{-1}AT_{\lambda}\| - \|-N\|| \leq \|(-\lambda I -
T_{\lambda}^{-1}AT_{\lambda}) -(-N)\|$, we obtain:
\begin{equation*}
0 \leq F(\lambda) \leq \frac{\sqrt{\sigma (2\lambda -1
    -2\|N\|)}}{|\|-\lambda I - T_{\lambda}^{-1}AT_{\lambda}\| -
  \|-N\||}\,.
\end{equation*} 
Computing the asymptotic limit of the latter inequality, noting that
$N$ is a constant matrix, and $T_{\lambda}^{-1} T_{\lambda} = I$,
yields:
\begin{equation*}
  0 \leq \lim_{\lambda \rightarrow \infty}{F(\lambda)} \leq
  \lim_{\lambda \rightarrow \infty} {\frac{\sqrt{2\sigma
        \lambda}}{\|T_{\lambda}^{-1}(-\lambda I - A)T_{\lambda}\|}}\,.
\end{equation*} 
On the other hand, for a matrix $\bar{A}\in\real^{n \times n}$ it
holds that $\rho(\bar{A}) \leq \|\bar{A}\|$, where $\rho(\bar{A})$ is
the spectral radius of $\bar{A}$; for further insight, refer
to~\cite{DSB:05}. Hence, the latter equation can be further bounded as
in the following:
\begin{equation}
  \label{eqn-claim-thm-tau-f-lambda-limit-bnd-1}
  0 \leq \lim_{\lambda \rightarrow \infty}{F(\lambda)} \leq
  \lim_{\lambda \rightarrow \infty} {\frac{\sqrt{2\sigma
        \lambda}}{\rho(T_{\lambda}^{-1}(\lambda I +
      A)T_{\lambda})}}\,.
\end{equation}
Furthermore, recalling that $T_{\lambda}^{-1}(\lambda I +
A)T_{\lambda}$ is the similarity transformation of the matrix $\lambda
I + A$, it holds that $\rho(T_{\lambda}^{-1}(\lambda I +
A)T_{\lambda}) = \rho(\lambda I + A)$. Accordingly, we get:
\begin{equation*}
  0 \leq \lim_{\lambda \rightarrow \infty}{F(\lambda)} \leq
  \lim_{\lambda \rightarrow \infty} {\frac{\sqrt{2\sigma
        \lambda}}{\rho(\lambda I + A)}}\,.
\end{equation*} 
Now, by means of the Ger$\check{\text{s}}$gorin disc theorem, we
obtain that:
\begin{align*}
 \rho(\lambda I & + A) \in \cup_{i=1}^{n} D(\lambda +
 a_{ii},\sum_{j\neq i}{|a_{ij}|})
% \cup_{i=1}^{n} {[-\lambda -a_{ii}-\sum_{j\neq i}{|a_{ij}|} , -\lambda-a_{ii}+\sum_{j\neq i}{|a_{ij}|}]}
\,.
\end{align*}
This implies that $\rho(\lambda I + A)$ has a linear growth as
$\lambda$ tends to infinity. Therefore, without loss of generality, we
can write the equation~\eqref{eqn-claim-thm-tau-f-lambda-limit-bnd-1}
as follows:
\begin{equation}
  \label{eqn-claim-thm-tau-f-lambda-limit-bnd-2}
  0 \leq \lim_{\lambda \rightarrow \infty}{F(\lambda)} \leq
  \lim_{\lambda \rightarrow \infty} {\frac{\sqrt{2\sigma
        \lambda}}{\lambda + c}}\,,
\end{equation}
where $c$ is some constant. Hence, based
on~\eqref{eqn-claim-thm-tau-f-lambda-limit-bnd-2}, it is easy to check
that~\eqref{eqn-claim-thm-tau-f-lambda-limit} holds, which completes
the proof of this claim.  \oprocend
%%%%%%%%%%%%%%%

Having proved $\lim_{\lambda \rightarrow \infty}{F(\lambda)} = 0$, we
would like to show $\lim_{\lambda \rightarrow \infty}{t_\lambda} = 0$,
which then, by the squeeze theorem,
   yields
$\lim_{\lambda \rightarrow \infty}{\tau_\lambda} = 0$. The next claim
goes along this direction.

%%%%%% claim %%%%%%
\begin{claim}\label{claim-in-thm-tau-t-lambda}
  The following holds:
\begin{equation}
  \label{eqn-claim-thm-tau-t-lambda-limit}
  \lim_{\lambda \rightarrow \infty}{t_\lambda} = 0\,.
\end{equation}
\end{claim}
\emph{Proof of Claim~\ref{claim-in-thm-tau-t-lambda}:} Recalling the
triggering law and by construction of $t_{\lambda}$, the following
holds:
\begin{equation}
  \label{eqn-claim-thm-tau-t-lambda-limit-bnd-1}
  |e_{\lambda}(t_{\lambda})| = F(\lambda) x_{\lambda}(t_{\lambda})|\,.
\end{equation}
On the other hand, recall that for $t \in [0 , t_{\lambda}]$:
\begin{equation*}
e_{\lambda}(t) = T_{\lambda}^{-1} e(t) = T_{\lambda}^{-1} (x(t) - x_{0})\,,
\end{equation*}
that yields:
\begin{equation*}
e_{\lambda}(t)= x_{\lambda}(t) - T_{\lambda}^{-1}x_{0}\,.
\end{equation*}
Applying the latter equation
on~\eqref{eqn-claim-thm-tau-t-lambda-limit-bnd-1} gives:
\begin{equation}
\label{eqn-claim-thm-tau-t-lambda-limit-bnd-2}
|x_{\lambda}(t_{\lambda}) - T_{\lambda}^{-1}x_{0}| = F(\lambda) |x_{\lambda}(t_{\lambda})|\,.
\end{equation}
In order to prove the result, we now consider two cases:

  \textbf{Case (i): $\mathbf{|x_{\lambda}(t_{\lambda})| = 0}$.}
In this case, by~\eqref{eqn-claim-thm-tau-t-lambda-limit-bnd-2}, it
holds that $x_{0} = 0$. Also since the system is linear and the chosen
control law of the form $u = Kx_{0}$, the system does not evolve. In
other words there is no need for triggering, $t_{\lambda} = 0$ which
renders $\tau_{\lambda} = 0$.

\textbf{Case (ii): $\mathbf{|x_{\lambda}(t_{\lambda})| \neq 0}$.} In
this case, dividing~\eqref{eqn-claim-thm-tau-t-lambda-limit-bnd-2} by
$|x_{\lambda}(t_{\lambda})|$, bestows:
\begin{equation*}
  F(\lambda) = \frac{|x_{\lambda}(t_{\lambda}) -
    T_{\lambda}^{-1}x_{0}|}{|x_{\lambda}(t_{\lambda})|}\,.
\end{equation*}
Computing the absolute-value of the latter equation, and applying the
triangular inequality, yields the following result:
\begin{align*}
  |F(\lambda)| = F(\lambda) & = \left| \frac{|x_{\lambda}(t_{\lambda})
    - T_{\lambda}^{-1}x_{0}|}{|x_{\lambda}(t_{\lambda})|} \right| \\ &
  \geq \left| \frac{|x_{\lambda}(t_{\lambda})| -
    |T_{\lambda}^{-1}x_{0}|}{|x_{\lambda}(t_{\lambda})|} \right| \geq
  0\,.
\end{align*}
Now, according to the Claim~\ref{claim-in-thm-tau-f-lambda} and by the
squeeze theorem, we obtain:
\begin{equation}
  \label{eqn-claim-thm-tau-t-lambda-limit-bnd-3}
  \lim_{\lambda \rightarrow \infty} {\left| 1 -
    \frac{|T_{\lambda}^{-1}x_{0}|}{|x_{\lambda}(t_{\lambda})|}
    \right|} = 0 \Rightarrow \lim_{\lambda \rightarrow \infty} {
    \frac{|x_{\lambda}(0)|}{|x_{\lambda}(t_{\lambda})|} } = 1\,.
\end{equation}
Now we will show that $\lim_{\lambda \rightarrow \infty}t_{\lambda} =
0$ exploiting a contradiction argument. Hence, let us assume that
$\lim_{\lambda \rightarrow \infty} {t_{\lambda}} \neq 0$. Then, the
negation of the mathematical definition of $\lim_{\lambda \rightarrow
  \infty} {t_{\lambda}} = 0$
   implies:
\begin{equation}
  \label{eqn-claim-thm-tau-t-lambda-limit-bnd-4}
  \exists t^{*}\, \text{such that}\, \forall \lambda>0, \exists
  \bar{\lambda}>\lambda\, \text{such that}\, t_{\bar{\lambda}} >
  t^{*}\,.
\end{equation}
In~\eqref{eqn-claim-thm-tau-t-lambda-limit-bnd-4}, let us take the
sequence $\{\lambda_{k}\}_{k\in\mathbb{N}}$, where $\lambda_{k}
\rightarrow \infty$ as $k \rightarrow \infty$, and set $\bar{\lambda}
\equiv \lambda_{k}$.  Then we have the counterpart sequence
$\{t_{\lambda_{k}}\}_{k\in\mathbb{N}}$, for which we have that
$t_{\lambda_{k}} > t^{*}$. Because of this, we have that:
\begin{equation*}
  |x_{\lambda_{k}}(t_{\lambda_{k}})| \leq |x_{\lambda_{k}}(t^{*})|\,,
\end{equation*}
which follows from $\dot{V}(t) < 0$ for $t<t_{\lambda_{k}}$---by the
choice of our triggering law---and that $V(t) =
|x_{\lambda_{k}}(t)|^{2}$.  To obtain a deeper insight, we refer the
reader to the proof of Theorem~\ref{main-thm-arbit-jammer}. 
Based on this observation, we derive the following inequality:
\begin{equation}
  \label{eqn-claim-thm-tau-t-lambda-limit-bnd-5}
\lim_{k \rightarrow \infty}
    {\frac{|x_{\lambda_{k}}(0)|}{|x_{\lambda_{k}}(t^{*})|}} \leq
    \lim_{k \rightarrow \infty}
        {\frac{|x_{\lambda_{k}}(0)|}{|x_{\lambda_{k}}(t_{\lambda_{k}})|}}\,.
\end{equation}
Since $t^{*} < t_{\lambda_{k}}$, and by
Equation~\eqref{eqn-thm-trans-sys-state-bound}, we obtain:
\begin{equation*}
|x_{\lambda_{k}}(t^{*})| \leq |x_{\lambda_{k}}(0)| \exp{( -(1-\sigma)
  (2\lambda_{k}-1-2\|N\|) t^{*}/2 )}\,,
\end{equation*}
which yields:
\begin{equation}
  \label{eqn-claim-thm-tau-t-lambda-limit-bnd-6}
  \frac{|x_{\lambda_{k}}(0)|}{|x_{\lambda_{k}}(t^{*})|} \geq \exp{(
    (1-\sigma) (2\lambda_{k}-1-2\|N\|) t^{*}/2 )}\,.
\end{equation}
As $\sigma\in(0,1)$, by properly letting $\lambda_{k} \geq 1/2+\|N\|$
go to infinity, we obtain:
\begin{equation*}
  \lim_{k \rightarrow \infty} {
    \frac{|x_{\lambda_{k}}(0)|}{|x_{\lambda_{k}}(t^{*})|} }= \infty,
\end{equation*}
which induces:
\begin{equation*}
  \lim_{k \rightarrow \infty}
      {\frac{|x_{\lambda_{k}}(0)|}{|x_{\lambda_{k}}(t_{\lambda_{k}})|}}
      = \infty\,,
\end{equation*}
from~\eqref{eqn-claim-thm-tau-t-lambda-limit-bnd-5}. From here we
conclude that~\eqref{eqn-claim-thm-tau-t-lambda-limit-bnd-6} is in
contradiction with~\eqref{eqn-claim-thm-tau-t-lambda-limit-bnd-3}.
Therefore, it must be that $\lim_{\lambda \rightarrow \infty}
{t_{\lambda}} =0$.

Henceforth, we observe that for both \textbf{Case~(i)} and
\textbf{Case~(ii)}, the
Equation~\eqref{eqn-claim-thm-tau-t-lambda-limit} holds, thus it
completes the proof of this claim.  \oprocend

Finally, from the set of inequalities $0 \leq \tau_{\lambda} \leq
t_{\lambda}$, which holds for all $\lambda \in \realpositive$, and by
Claim~\ref{claim-in-thm-tau-t-lambda} that shows that $\lim_{\lambda
  \rightarrow \infty} {t_{\lambda}} = 0$, we can conclude that
$\lim_{\lambda \rightarrow \infty} {\tau_{\lambda}} = 0$.
\end{proof}

In this paper, we assume that the jammer is imposing a
\emph{``worst-case jamming scenario''}, i.e.,
$\subscr{T}{off}=\supscr{\subscr{T}{off}}{cr}$. Now, having
established the Proposition~\ref{prop-jordan-trig-seq-event-trig-cond}
and introduced the parameter $\tau_{\lambda}$ in 
Remark~\ref{remark-no-accum-event-trig}, we define the triggering
strategy as follows.
\begin{definition}
\label{defn-jordan-trig-trig-law}
  The triggering strategy used
in this paper, despite presence of the jammer, is as follows:
\begin{align}                                                               
  \label{eqn-jordan-trig-law-jammer}
  t^{*}_{k,n}\in& \left\{ l\tau_{\lambda} \ \big| l\tau_{\lambda} \in
  [(n-1)T, (n-1)T+ \supscr{\subscr{T}{off}}{cr}] \right\}\cup\nonumber
  \\ & \left\{ nT \right\}\,.
\end{align}

  We note that based
on Theorem~\ref{thm-jordan-trig-seq-tau-lambda}, and for a given $T$,
we can find a $\lambda_{c}$ so that the multiples of $\tau_{\lambda}$
lie in the desired interval, i.e., the set introduced
in~\eqref{eqn-jordan-trig-law-jammer} is never empty.  The triggering
law introduced in our recent paper~\cite{HSF-SM:12-cdc} has inspired
this strategy. The main difference between both laws is the choice of
triggering times. While here we adapt the triggering sequence via an
appropriate choice of $\tau_{\lambda}$ that depends on the jammer,
in~\cite{HSF-SM:12-cdc} we study when the time-sequence generated
by~\eqref{eqn-prop-event-trig-cond} is sufficient to beat a given
jammer.  

In~\eqref{eqn-jordan-trig-law-jammer}, $k\in\mathbb{N}$ denotes the
number of triggering times occurring in the $\supscr{n}{th}$ jammer
action-period, and $l\in\mathbb{N}$ stands for the multiples of
$\tau_{\lambda}$ starting from $l=1$ in the first period and adding up
afterwards.
\end{definition}

%%%%%%%%%%%%%%%%%%%%%%%%%%%%%%%%%%%%%%%%%%%%%%%%
\section{Stability Analysis of the Triggering Strategy}
\label{sec-stab-analysis}

  In this section, we
shall present the main result of this paper to guarantee the stability
of the class of systems considered under the given type of jamming
attacks.

 The following bound is found in~\cite{CVL:77}:
\begin{equation}
  \label{eqn-bound-exp-matrix-dummy}
  \left\|\exp(M)\right\| \leq \exp(\mu(M))\,, \quad  M\in\real^{n\times n}\,,
\end{equation}
where the $\mu$ operator is defined as follows:
\begin{equation}
\label{eqn-mu-oper-expr-dummy}
  \mu(M)=\max \left\{\mu \, \huge|
  \,\mu\in\spec\left(\frac{M+M^{T}}{2}\right)\right\}\,.
\end{equation}
In the proof of next result and in order to avoid the sign confusion,
we shall use a variation of~\eqref{eqn-mu-oper-expr-dummy}.  Denote:
\begin{equation}
\label{eqn-mu-oper-expr}
\mu_M \triangleq |\mu(M)| + 1\,. 
\end{equation}
Then, the following holds:
\begin{equation}
  \label{eqn-bound-exp-matrix}
  \left\|\exp(M)\right\| \leq \exp(\mu_M)\,.
\end{equation}

\begin{theorem}
\label{main-thm-arbit-jammer}
Consider the system~\eqref{eqn-sys-matrix-prel} of order lesser than
$5$, where $(A,B)$ is a controllable pair. Given a jammer signal
$(\supscr{\subscr{T}{off}}{cr}, T)$, then $\exists
\lambda^{*}>\|N\|+1/2$, such that $\forall \lambda \geq \lambda^{*}$,
the system with control gain $K_{\lambda}$ as chosen in
Proposition~\ref{prop-lambda-eigenvalue-prel} and with triggering
strategy~\eqref{eqn-jordan-trig-law-jammer}, is asymptotically stable.
\end{theorem}
\begin{proof}
  We shall focus on the first jammer action-period, i.e., $0 \leq t
  \leq T$. For the sake of brevity, we drop $n=1$ in the $t^{*}_{k,n}$
  annotation. Without loss of generality, let $t^{*}_{k}=
  k\tau_{\lambda}$, for $k\in\until m$, be the time-sequence generated
  by~\eqref{eqn-jordan-trig-law-jammer}, where $m$ is such
  that:

  \begin{equation*}
    t^{*}_{m}=m\tau_{\lambda} \leq \supscr{\subscr{T}{off}}{cr} <
    t^{*}_{m+1} = (m+1)\tau_{\lambda}\,.
  \end{equation*}
  We note that we can always assume this, since according to the
  Theorem~\ref{thm-jordan-trig-seq-tau-lambda}, we can make
  $\tau_{\lambda}$ arbitrarily small by choosing $\lambda$ large
  enough. As $\tau_{\lambda} > 0$, the latter equation yields:
  \begin{equation*}
    m \leq \frac{ \supscr{\subscr{T}{off}}{cr} }{\tau_{\lambda}} < m+1\,.
  \end{equation*}
  Thus, $\lfloor \frac{ \supscr{\subscr{T}{off}}{cr} }{\tau_{\lambda}}
  \rfloor = m$, where $\lfloor . \rfloor$ is the floor operator, and
  \begin{equation}
    \label{eqn-floor-tm-tau-lambda-dummy}
    t^{*}_{m} = \left\lfloor \frac{ \supscr{\subscr{T}{off}}{cr}
      }{\tau_{\lambda}} 
    \right\rfloor \tau_{\lambda}\,.
  \end{equation}
 
  It is easy to see that for all $a>0$, if $\lfloor a \rfloor \geq 1$,
  then $\lfloor a \rfloor \geq a/2$. Based on this observation, and as
  $\frac{ \supscr{\subscr{T}{off}}{cr} }{\tau_{\lambda}} \geq 1$, then
  it holds that $\lfloor \frac{ \supscr{\subscr{T}{off}}{cr}
  }{\tau_{\lambda}} \rfloor \geq \frac{ \supscr{\subscr{T}{off}}{cr}
  }{2\tau_{\lambda}}$, which by~\eqref{eqn-floor-tm-tau-lambda-dummy},
  gives:
  \begin{equation}
    \label{eqn-floor-tm-tau-lambda}
    t^{*}_{m} = \left\lfloor \frac{ \supscr{\subscr{T}{off}}{cr}
      }{\tau_{\lambda}} \right\rfloor 
    \tau_{\lambda} \geq \frac{\supscr{\subscr{T}{off}}{cr}}{2}\,.
  \end{equation}
  The rest of the proof goes over the following steps:
  \begin{enumerate}
  \item We break the time-interval $[0, T]$ into two subintervals $[0,
    t^{*}_{m+1}]$, and $[t^{*}_{m+1}, T]$; namely, when the jammer is
    \textit{sleeping} and \textit{active}, respectively,
    
  \item Then, in order to find an estimate for $|x(t^{*}_{m+1})|$, and
    $|x(T)|$, we first transform the original system into new
    coordinates by the matrix $T_{\lambda}$; we perform some
    computations, and transform it back into its original coordinates,
    by $T_{\lambda}^{-1}$ --- this is done for each subinterval $[0,
      t^{*}_{m+1}]$, and $[t^{*}_{m+1}, T]$.  This way, the analysis
    becomes more tractable, 
  
  \item Finally, the theorem conclusion will follow by studying a
    coefficient $C(\lambda)$, appearing in $|x(T)| < C(\lambda)
    |x(0)|$. Due to $\lim_{\lambda \rightarrow \infty} {C(\lambda)} =
    0$, we will be able to guarantee $\{|x(nT)|\}$ is decreasing and
    use a Lyapunov argument to prove stability.
  \end {enumerate}
     
  Let us consider the transformed
  system~\eqref{eqn-jordan-triggering-trans-sys}. We observe that for
  $t \in [0, t^{*}_{m+1}]$, and according to
  Remark~\ref{remark-no-accum-event-trig}, the
  event~\eqref{eqn-prop-event-trig-cond} introduced in
  Proposition~\ref{prop-jordan-trig-seq-event-trig-cond} holds, and
  that also that $V = x^{T}_{\lambda}x_{\lambda} = |x_{\lambda}|^{2}$ is an
  ISS-Lyapunov function. Hence, resorting to the proof of this
  proposition, the following inequality holds:
  \begin{align*}
    \dot{V}(x_{\lambda}) \leq &-(1-\sigma)(2\lambda -1 -2\|N\|)
    |x_{\lambda}|^{2} =\\ & -(1-\sigma)(2\lambda -1 -2\|N\|)
    V(x_{\lambda}), \; \forall t\in[0, t^{*}_{m+1}]\,.
  \end{align*}

The latter equation, by the comparison principle, yields:
  \begin{equation*}
    V(x_{\lambda}) \leq V(x_{\lambda}(0)) \exp{(-(1-\sigma)(2\lambda
      -1 -2\|N\|)t)}\,.
  \end{equation*}
  which then, by recalling $V = x^{T}_{\lambda}x_{\lambda} =
  |x_{\lambda}|^{2}$, yields:
  \begin{equation}
    \label{eqn-thm-trans-sys-state-bound}
    |x_{\lambda}(t)| \leq |x_{\lambda}(0)| \exp{(-(1-\sigma)(2\lambda
      -1 -2\|N\|)t/2)}\,.
  \end{equation}
  Now, we have to transform the latter equation into original
  coordinates. First, by using $x(t) = T_{\lambda}x_{\lambda}(t)$:
  \begin{equation}
    \label{eqn-thm-trans-sys-state-bound-transformer}
    \lambda_{\min}((T_{\lambda}^{-1})^{T} (T_{\lambda}^{-1})) |x|^{2}
    \leq |x_{\lambda}|^{2} \leq \|T_{\lambda}^{-1}\|^{2} |x|^{2}\,.
  \end{equation}
  The latter equation is obtained noting that (i) $|x_{\lambda}|^{2} =
  x^{T} (T_{\lambda}^{-1})^{T} (T_{\lambda}^{-1}) x$, and (ii) the
  matrix $(T_{\lambda}^{-1})^{T} (T_{\lambda}^{-1})$ is a
  positive-definite symmetric matrix.

  According to~\eqref{eqn-thm-trans-sys-state-bound-transformer},
  Equation~\eqref{eqn-thm-trans-sys-state-bound} implies:
  \begin{equation}
    \label{eqn-thm-orig-sys-state-bound-whole-time-span}
    |x(t^{*}_{m})| \leq \frac{\|T_{\lambda}^{-1}\|
      \exp{(-(1-\sigma)(2\lambda-1-2\|N\|)t^{*}_{m}/2)}}
    {\sqrt{\lambda_{\min}((T_{\lambda}^{-1})^{T} (T_{\lambda}^{-1}))}}
    |x_{0}|\,,
\end{equation}
  which is computed for $t= t^{*}_{m}$.
     
     In addition, in an analogous way, this time considering
     $t\in[t^{*}_{m}, t^{*}_{m+1}]$, 
     we can obtain the following result:
  \begin{align}
    \label{eqn-thm-orig-sys-state-bound-last-segment}
    |x(t^{*}_{m+1})| \leq &\frac{\|T_{\lambda}^{-1}\|
      \exp{(-(1-\sigma)(2\lambda-1-2\|N\|)\tau_{\lambda}/2)}}
    {\sqrt{\lambda_{\min}((T_{\lambda}^{-1})^{T} (T_{\lambda}^{-1}))}}
    \times \nonumber \\ & |x(t^{*}_{m})|\,,
  \end{align}
where we note that $\tau_{\lambda}$ appears, as by our triggering law,
$t^{*}_{m+1} - t^{*}_{m} = \tau_{\lambda}$.  

   Let us consider the
transformed system~\eqref{eqn-jordan-triggering-trans-sys}, once
more. We consider the time-interval $[t^{*}_{m+1}, T]$, then
$e_{\lambda}(t)= x_{\lambda}(t^{*}_{m}) - x_{\lambda}(t)$ and so an
equivalent form of~\eqref{eqn-jordan-triggering-trans-sys} can be
written as:
\begin{equation*}
  \dot{x}_{\lambda} = T_{\lambda}^{-1} A T_{\lambda} x_{\lambda} +
  T_{\lambda}^{-1} BK_{\lambda} T_{\lambda} x_{\lambda}(t^{*}_{m}),
  \quad \forall t\in [t^{*}_{m+1}, T]\,.
\end{equation*}
Solving this dynamics for the initial condition
$x_{\lambda}(t^{*}_{m+1})$, we obtain the following:
\begin{align}
  \label{eqn-thm-soln-trans-dynamics}
  x&_{\lambda}(t) = \exp{( (t-t^{*}_{m+1}) T_{\lambda}^{-1} A
    T_{\lambda} )} x_{\lambda}(t^{*}_{m+1}) +
  \nonumber\\ &\int^{t}_{t^{*}_{m+1}} {\exp{( (t-s) T_{\lambda}^{-1} A
      T_{\lambda} )} T_{\lambda}^{-1} B K_{\lambda} T_{\lambda}
    x_{\lambda}(t^{*}_{m}) \, \text{d}s}\,,
\end{align}

  which holds for $t\in [t^{*}_{m+1}, T]$. 
  In order to further simplify the latter equation, we use the fact that for a given matrix $A\in\real^{n\times n}$, 
  and invertible matrix
  $T\in\real^{n\times n}$, it holds: $\exp{(T^{-1}AT) } = T^{-1}
  \exp{(A)} T$.

Hence, 
Equation~\eqref{eqn-thm-soln-trans-dynamics} is simplified as follows:
\begin{align*}
  x_{\lambda}(t) =& T_{\lambda}^{-1} \exp{( (t-t^{*}_{m+1}) A )}
  T_{\lambda} x_{\lambda}(t^{*}_{m+1}) + \\ &\int^{t}_{t^{*}_{m+1}}
  {T_{\lambda}^{-1} \exp{( (t-s) A )} B K_{\lambda} T_{\lambda}
    x_{\lambda}(t^{*}_{m}) \, \text{d}s}\,,
\end{align*}
which then results in the following equation:
\begin{align*}
  T_{\lambda} x_{\lambda}(t) =& \exp{( (t-t^{*}_{m+1}) A )}
  T_{\lambda} x_{\lambda}(t^{*}_{m+1}) + \\ &\int^{t}_{t^{*}_{m+1}}
  {\exp{( (t-s) A )} B K_{\lambda} T_{\lambda} x_{\lambda}(t^{*}_{m})
    \, \text{d}s}\,,
\end{align*}
and using $x = T_{\lambda}x_{\lambda}$ to transform it back into the
original dynamics, yields:
\begin{align}
  \label{eqn-thm-soln-orig-dynamics}
  x(t) =& \exp{( (t-t^{*}_{m+1}) A )} x(t^{*}_{m+1}) +
  \nonumber\\ &\int^{t}_{t^{*}_{m+1}} {\exp{( (t-s) A )} B K_{\lambda}
    x(t^{*}_{m}) \, \text{d}s}\,.
\end{align}
We upper-bound~\eqref{eqn-thm-soln-orig-dynamics},
using~\eqref{eqn-bound-exp-matrix}, which results in the following:
\begin{align*}
  |x(t)| \leq& |x(t^{*}_{m+1})| \exp{((t-t^{*}_{m+1})\mu_A)} +
  \\ &|x(t^{*}_{m})| \|BK_{\lambda}\| \int^{t}_{t^{*}_{m+1}}
     {\exp{((t-s)\mu_A)}\, \text{d}s}\,.
\end{align*}
We evaluate the latter equation at $t=T$, and then solve the integral
to obtain:
\begin{align}
  \label{eqn-thm-ineq-orig-dynamics-1}
  |x(T)| \leq& |x(t^{*}_{m+1})| \exp{((T-t^{*}_{m+1})\mu_A)} +
  \nonumber\\ &|x(t^{*}_{m})| \frac{\|BK_{\lambda}\|}{\mu_A}
  (\exp{((T-t^{*}_{m+1})\mu_A)-1)}\,.
\end{align}
Since $T - \supscr{\subscr{T}{off}}{cr} =
\supscr{\subscr{T}{on}}{cr}$, and $\supscr{\subscr{T}{off}}{cr} <
t^{*}_{m+1}$, we have that:
\begin{equation*}
  T - t^{*}_{m+1} < \supscr{\subscr{T}{on}}{cr}\,.
\end{equation*}
Thus, we can rewrite~\eqref{eqn-thm-ineq-orig-dynamics-1} as:
\begin{align}
  \label{eqn-thm-ineq-orig-dynamics-2}
  |x(T)| \leq & |x(t^{*}_{m+1})| \exp{(\supscr{\subscr{T}{on}}{cr}
    \mu_A)} + \nonumber\\ & |x(t^{*}_{m})|
  \frac{\|BK_{\lambda}\|}{\mu_A} (\exp{(\supscr{\subscr{T}{on}}{cr}
    \mu_A)-1)}\,.
\end{align}
Applying now
Equation~\eqref{eqn-thm-orig-sys-state-bound-last-segment}
on~\eqref{eqn-thm-ineq-orig-dynamics-2}, we get:
\begin{align}
  \label{eqn-thm-ineq-orig-dynamics-3}
  & \frac{|x(T)|}{|x(t^{*}_{m})|} \leq \bigg(
  \frac{\|BK_{\lambda}\|}{\mu_{A}} (\exp{(\supscr{\subscr{T}{on}}{cr}
    \mu_A)} -1) + \nonumber \\ &\frac{ \exp{(-(1-\sigma)(2\lambda -1
      -2\|N\|) \tau_{\lambda}/2)}}
       {\|T_{\lambda}^{-1}\|^{-1}\sqrt{\lambda_{\min}((T_{\lambda}^{-1})^{T}
           (T_{\lambda}^{-1}))}} \exp{(\supscr{\subscr{T}{on}}{cr}
         \mu_A)} \bigg)\,.
\end{align}
Now, combining~\eqref{eqn-thm-orig-sys-state-bound-whole-time-span}
and~\eqref{eqn-thm-ineq-orig-dynamics-3}, we obtain:
\begin{align}
  \label{eqn-thm-ineq-orig-dynamics-4}
&\frac{|x(T)|}{|x_{0}|} \leq \left( \frac{ \exp{(-(1-\sigma)(2\lambda
      -1 -2\|N\|)t^{*}_{m}/2)}} {\|T_{\lambda}^{-1}\|^{-1}
    \sqrt{\lambda_{\min}((T_{\lambda}^{-1})^{T} (T_{\lambda}^{-1}))}}
  \right) \times \nonumber\\ &\bigg( \frac{\|BK_{\lambda}\|}{\mu_{A}}
  (\exp{(\supscr{\subscr{T}{on}}{cr} \mu_A)} -1) + \nonumber
  \\ &\frac{ \exp{(-(1-\sigma)(2\lambda -1 -2\|N\|) \tau_{\lambda})}}
     {\|T_{\lambda}^{-1}\|^{-1}
       \sqrt{\lambda_{\min}((T_{\lambda}^{-1})^{T}
         (T_{\lambda}^{-1}))}} \exp{(\supscr{\subscr{T}{on}}{cr}
       \mu_A)} \bigg)\,.
\end{align}
To obtain the main equation, we shall
use~\eqref{eqn-floor-tm-tau-lambda} to further
bound~\eqref{eqn-thm-ineq-orig-dynamics-4}, which then results in:
\begin{align}
  \label{eqn-thm-ineq-orig-dynamics-main}
  &\frac{|x(T)|}{|x_{0}|} \leq \left( \frac{
    \exp{(-(1-\sigma)(2\lambda -1
      -2\|N\|)\supscr{\subscr{T}{off}}{cr}/4)}} {
    \|T_{\lambda}^{-1}\|^{-1}
    \sqrt{\lambda_{\min}((T_{\lambda}^{-1})^{T} (T_{\lambda}^{-1}))}}
  \right) \times \nonumber\\ &\bigg( \frac{\|BK_{\lambda}\|}{\mu_{A}}
  (\exp{(\supscr{\subscr{T}{on}}{cr} \mu_A)} -1) + \nonumber
  \\ &\frac{\exp{(-(1-\sigma)(2\lambda -1 -2\|N\|) \tau_{\lambda})}}
     {\|T_{\lambda}^{-1}\|^{-1}
       \sqrt{\lambda_{\min}((T_{\lambda}^{-1})^{T}
         (T_{\lambda}^{-1}))}} \exp{(\supscr{\subscr{T}{on}}{cr}
       \mu_A)} \bigg) \triangleq \nonumber\\ & C(\lambda)\,.
\end{align}

  We present now the following result on the
coefficient $C(\lambda)$, introduced in the latter inequality.

\begin{claim}
\label{claim-lim-c-lambda}
In~\eqref{eqn-thm-ineq-orig-dynamics-main} following holds:
\begin{equation}
\label{eqn-claim-thm-main-assertion}
\lim_{\lambda \rightarrow \infty} C(\lambda) = 0\,.
\end{equation}
\end{claim}
\emph{Proof of Claim~\ref{claim-lim-c-lambda}:} In order to complete
the proof, we shall break $C(\lambda)$-expression as follows:
\begin{equation*}
  C(\lambda) = C_{1}(\lambda) (C_{2}(\lambda) + C_{3}(\lambda))\,,
\end{equation*}
where:
\begin{equation*}
  C_{1}(\lambda) = \left( \frac{ \exp{(-(1-\sigma)(2\lambda -1
      -2\|N\|)\supscr{\subscr{T}{off}}{cr}/4)}}
  {\|T_{\lambda}^{-1}\|^{-1}
    \sqrt{\lambda_{\min}((T_{\lambda}^{-1})^{T} (T_{\lambda}^{-1}))}}
  \right)\,,
\end{equation*}
\begin{equation*}
  C_{2}(\lambda) = ( \frac{\|BK_{\lambda}\|}{\mu_{A}}
  (\exp{(\supscr{\subscr{T}{on}}{cr} \mu_A)} -1)\,,
\end{equation*}
and
\begin{equation*}
  C_{3}(\lambda) = \frac{ \exp{(-(1-\sigma)(2\lambda -1 -2\|N\|)
      \tau_{\lambda})}} {\|T_{\lambda}^{-1}\|^{-1}
    \sqrt{\lambda_{\min}((T_{\lambda}^{-1})^{T} (T_{\lambda}^{-1}))}}
  \exp{(\supscr{\subscr{T}{on}}{cr} \mu_A)}\,.
\end{equation*}
Then, we shall show that $\lim_{\lambda \rightarrow \infty}
{C_{1}(\lambda)C_{2}(\lambda)} = 0$, and $\lim_{\lambda \rightarrow
  \infty} {C_{1}(\lambda)C_{3}(\lambda)} = 0$.

According to~\eqref{eqn-jordan-decomp-closed-loop-matrix}, and
recalling $J_{\lambda} = -\lambda I + N$, we get $BK_{\lambda} = -A +
T_{\lambda}^{-1} (-\lambda I + N) T_{\lambda}$, which then results in
$BK_{\lambda} = -A - \lambda I + T_{\lambda}^{-1} N
T_{\lambda}$. Therefore, from this matrix equality and applying the 2-norm
operator on both sides, we get $\|BK_{\lambda}\| = \|-(A+\lambda I) +
T_{\lambda}^{-1} N T_{\lambda}\|$, applying triangular-inequality on
the right-hand side, we get $\|BK_{\lambda}\| \leq \| -(A+\lambda I)\|
+ \|T_{\lambda}^{-1} N T_{\lambda}^{-1}\|$, which can be further
upper-bounded as follows:
\begin{equation*}
  \|BK_{\lambda}\| \leq \|A\| + |\lambda| + \|T_{\lambda}^{-1}\| \|N\|
  \|T_{\lambda}\|\,.
\end{equation*}
We shall employ this latter inequality, in order to obtain a new
upper-bound for $C_{1}(\lambda) C_{2}(\lambda)$:
\begin{align}
  \label{eqn-claim-thm-new-upp-bound-c1-c2}
  0 \leq & C_{1}(\lambda)C_{2}(\lambda) \leq C_{1}(\lambda) \times
  \nonumber\\ &\left(\frac{\|A\| +|\lambda| +\|T_{\lambda}^{-1}\|
    \|N\| \|T_{\lambda}} {\mu_A} (\exp{(\supscr{\subscr{T}{on}}{cr}
    \mu_A)} -1) \right)\,.
\end{align}
Now, in order to show that $\lim_{\lambda \rightarrow \infty}
C_{1}(\lambda) C_{2}(\lambda) = 0$, we put together these two facts:
(i) $C_{1}(\lambda)$ decays exponentially, as $\lambda \rightarrow
\infty$, since $\lambda > \|N\| +1/2$, and $\sigma\in(0,1)$, (ii)
based on what we explained in
Section~\ref{sec-jordan-decomp-triggering}, the matrices
$T_{\lambda}$, and $T_{\lambda}^{-1}$ depend on $\lambda$ in a
rational way, so the values $\|T_{\lambda}\|$, and
$\|T_{\lambda}^{-1}\|$ depend on $\lambda$ in a semi-algebraic
form~\cite{SB-RP-MFR:06}, thus the dependency of the coefficient of
$C_{1}(\lambda)$ appearing in upper-bound of $C_{1}(\lambda)
C_{2}(\lambda)$ in~\eqref{eqn-claim-thm-new-upp-bound-c1-c2} on
$\lambda$ is of a semi-algebraic form, which is dominated by an
exponential dependency. Therefore, we can conclude that the
upper-bound of $C_{1}(\lambda) C_{2}(\lambda)$
in~\eqref{eqn-claim-thm-new-upp-bound-c1-c2}, tends to zero as
$\lambda \rightarrow \infty$. Henceforth, as the lower-bound of
$C_{1}(\lambda) C_{2}(\lambda)$ is zero, then we conclude that:
\begin{equation}
  \label{eqn-claim-thm-c1-c2-final}
  \lim_{\lambda \rightarrow \infty} {C_{1}(\lambda) C_{2}(\lambda)} =
  0\,.
\end{equation}
In the following lines, we show that $\lim_{\lambda \rightarrow
  \infty} {C_{1}(\lambda)C_{3}(\lambda)} = 0$. First, we note that $0
\leq \tau_{\lambda} \leq T$, and $-(1-\sigma) (2\lambda -1 -2\|N\|)
\leq 0$, therefore, we get the following bounds for $C_{2}(\lambda)$:
\begin{align*}
  & \frac{\|T_{\lambda}^{-1}\| \exp{(-(1-\sigma)(2\lambda -1
      -2\|N\|)T/2)}}{\sqrt{\lambda_{\min}((T_{\lambda}^{-1})^{T}(T_{\lambda}^{-1}))}}
  \exp{(\supscr{\subscr{T}{on}}{cr} \mu_A)} \leq \\ & C_{3}(\lambda)
  \leq
  \frac{\|T_{\lambda}^{-1}\|}{\sqrt{\lambda_{\min}((T_{\lambda}^{-1})^{T}(T_{\lambda}^{-1}))}}
  \exp{(\supscr{\subscr{T}{on}}{cr} \mu_A)}\,.
\end{align*}
Also, $C_{1}(\lambda) > 0, \forall \lambda$, hence we can multiply the
latter inequality by $C_{1}(\lambda)$:
\begin{align}
  \label{eqn-claim-thm-main-bounds-c1-c3}
  & \frac{\|T_{\lambda}^{-1}\| \exp{(-(1-\sigma)(2\lambda -1
      -2\|N\|)T/2)}}{\sqrt{\lambda_{\min}((T_{\lambda}^{-1})^{T}(T_{\lambda}^{-1}))}
    \exp{(-\supscr{\subscr{T}{on}}{cr} \mu_A)}} C_{1}(\lambda) \leq
  \nonumber\\ & C_{1}(\lambda) C_{3}(\lambda) \leq
  \frac{\|T_{\lambda}^{-1}\| \exp{(\supscr{\subscr{T}{on}}{cr}
      \mu_A)}}{\sqrt{\lambda_{\min}((T_{\lambda}^{-1})^{T}(T_{\lambda}^{-1}))}}
  C_{1}(\lambda)\,.
\end{align}
Then, we study the limit of upper- and lower-bounds
of~\eqref{eqn-claim-thm-main-bounds-c1-c3}. Let us plug
$C_{1}(\lambda)$-expression in the lower-bound
of~\eqref{eqn-claim-thm-main-bounds-c1-c3}, we obtain: 

\begin{align*}
  & \text{LB}_{C_{1}C_{3}}(\lambda) \triangleq \\ & \frac{
    \|T_{\lambda}^{-1}\|^{2} \exp{(-(1-\sigma)(2\lambda -1
      -2\|N\|)(T/2+\supscr{\subscr{T}{off}}{cr}/4))} } {\lambda_{\min}
    ((T_{\lambda}^{-1})^{T} (T_{\lambda}^{-1}) )}\,.
\end{align*}
In order to show that $\lim_{\lambda \rightarrow \infty}
{\text{LB}_{C_{1}C_{3}}(\lambda)} = 0$, we recall three facts: (i) Since
$\sigma\in(0,1)$, $\lambda > \|N\| + 1/2$, then there is an
exponentially decaying term in $LB_{C_{1}C_{3}}(\lambda)$. (ii)
Recalling our discussion in
Section~\ref{sec-jordan-decomp-triggering}, $T_{\lambda}^{-1}$ depends
on $\lambda$ in a semi-algebraic way, which is dominated by
exponential decay. (iii) Once again, referring to
Section~\ref{sec-jordan-decomp-triggering}, the matrix
$(T_{\lambda}^{-1})^{T} (T_{\lambda}^{-1})$ depends on $\lambda$ in a
rational way. Hence, its characteristic polynomial depends on this
parameter in a rational way, moreover, we note that this polynomial is
of degree $4$ or less, by assumption, and due to Galois theory, the
dependency of the roots of this polynomial---including $\lambda_{\min}
((T_{\lambda}^{-1})^{T} (T_{\lambda}^{-1}) )$---
  on $\lambda$ is of semi-algebraic form, which is
dominated by exponential decay.

Having discussed the behavior at infinity of the lower-bound
of~\eqref{eqn-claim-thm-main-bounds-c1-c3}, we study the behavior of
its upper-bound at infinity. Let us plug the $C_{1}(\lambda)$
expression in the upper-bound
of~\eqref{eqn-claim-thm-main-bounds-c1-c3}. We then
obtain:
\begin{align*}
  & \text{UB}_{C_{1}C_{3}}(\lambda) \triangleq \\ & \frac{
    \|T_{\lambda}^{-1}\|^{2} \exp{(-(1-\sigma)(2\lambda -1
      -2\|N\|)(\supscr{\subscr{T}{off}}{cr}/4))} } {\lambda_{\min}
    ((T_{\lambda}^{-1})^{T} (T_{\lambda}^{-1}) )
    \exp{(-\supscr{\subscr{T}{on}}{cr} \mu_A)} }\,.
\end{align*}
Similar to $\text{LB}_{C_{1}C_{3}}(\lambda)$, it is easy to conclude that
$\lim_{\lambda \rightarrow \infty} {\text{UB}_{C_{1}C_{3}}(\lambda)} = 0$.
In the previous paragraphs, we have shown that the limit behavior as
$\lambda$ grows of the lower- and upper-bound
of~\eqref{eqn-claim-thm-main-bounds-c1-c3} is $0$. Hence, we infer
that:
\begin{equation}
  \label{eqn-claim-thm-c1-c3-final}
  \lim_{\lambda \rightarrow \infty} {C_{1}(\lambda) C_{3}(\lambda)} =
  0\,.
\end{equation}
Finally, having shown that the
Equations~\eqref{eqn-claim-thm-c1-c2-final},
and~\eqref{eqn-claim-thm-c1-c3-final}, hold, we have
proven~\eqref{eqn-claim-thm-main-assertion}. This completes the proof
of this claim.  \oprocend

At this stage, we have proven that in $|x(T)| \leq |x_{0}|
C(\lambda)$, it holds that $\lim_{\lambda \rightarrow \infty}
{C(\lambda)} = 0$. The main consequence of this conclusion is that,
based on the definition of limit, the following holds:
\begin{equation}
  \label{eqn-thm-lim-defn-c-lambda}
  \text{given}\, \epsilon >0, \exists \lambda^{*},\, \text{such
    that}\, \forall \lambda \geq \lambda^{*} \Rightarrow
  |C(\lambda)|<\epsilon\,,
\end{equation}
so, in other words, we can arbitrarily tune the
\textit{decaying-rate} of the states via $\lambda$ (and its effect on
$C(\lambda)$).  It is, nonetheless, worth mentioning that in order to
paraphrase the asymptotic stability, the parameter $\epsilon$ has to
be chosen such that $\epsilon < 1$, so that according
to~\eqref{eqn-thm-lim-defn-c-lambda}, $C(\lambda)<1$, and so $V(T) <
V(0)$, which ensures the asymptotic stability as demonstrated in Claim
4.3 in~\cite{HSF-SM:12-cdc}.
\end{proof}

\begin{remark}
\label{remark-main-thm-diag-dom}
As stated in the statement of this theorem, this is valid for systems
of order $4$, or less. An open question is the investigation the
classes of systems of higher order for which the result still holds.
In particular, we have observed the validity of the result for a
system of order $5$, and have presented the results in the
Section~\ref{sec-simulations}.
%This fact stems from what we have discussed in
%Claim~\ref{claim-lim-c-lambda} while finding the limit of $C(\lambda)$
%as $\lambda$ goes to infinty. \margin{are you now discussing other
%  alternatives? }\marginH{I'm trying to justify why we couldn't have
%  covered higher order systems.}Indeed, using
%Ger$\check{\text{s}}$gorin circle theorem in order to bound the
%minimum eigenvalue of the matrix $(T_{\lambda}^{-1})^{T}
%(T_{\lambda}^{-1})$ is not conclusive, because it comes with the
%valor\margin{what do you mean with valor?}\marginH{The only motivation
%  behind this line is to say that how Gershgorin result \emph{could}
%  have been used, and why it has not been. E.g, we can change valor to
%  price.} of the diagonally dominance assumption on this
%matrix. However, as highlighted in the simulation section, this
%additional assumption does not work out even for third-order
%systems.
%\margin{Since Gersgorin result does not help, then I wouldn't mention
%  it. Moreover, I would remove this remark, and discuss the potential
%  validity of the results for higher order systems.}  \marginH{I
%  changed this remark, according to this last comment of yours, how
%  does it sound?}\margin{OK, I've changed it slightly}
\end{remark}

In the following remark, we shall discuss that the results obtained so far are valid for an alternative
problem formulation:

\begin{remark}
\label{remark-main-thm-alt-prob-form}
We recall from the Section~\ref{sec-problem-formulation} that in the current problem formulation,
the jammer is corrupting the control channel signal, whereas the
observation channel is safe. Indeed, we would like to point out that the results presented
in this paper---including Theorem~\ref{main-thm-arbit-jammer}---
are still valid under the other problem formulation, i.e., the observation channel is also 
corrupted by the \emph{same} jammer. This is the case, because the measurement data is required at the same 
time-instant as the control is transmitted, which is successfully available, since the observation channel
is jammed by the same jammer. 
\end{remark}

\begin{remark}
It is worth noting that our discussion in the Remark~\ref{remark-main-thm-alt-prob-form} is no
longer valid for our previous results presented in~\cite{HSF-SM:12-cdc}, since therein the \emph{continuous} 
measurement of the states was necessary.
\end{remark}

%\marginH{the new remark on the resources idea}
\begin{remark}
At last, let us call (i) ``frequency of communication'' characterized by $\tau_{\lambda}$, and (ii) ``actuation effort''
characterized by $\lambda$ and $K_{\lambda}$, the two \textit{resources} that the operator posesses in order to counteract the
jammer. We would like to emphasize that in the method we are proposing, it is not feasible to decouple the utility of these two resources.
In better words, the coupling between the utility of these resources yields the main results presented thus far. We, nonetheless, do not deny that 
an alternative approach may exist which encompasses this decoupling idea and yields the same assertion as in Theorem~\ref{main-thm-arbit-jammer}.
\end{remark}

%\begin{remark}
%We categorize multi-input systems as (i) fully-actuated, and (ii)
%under-actuated systems. Regarding \textit{fully-actuated systems}, one
%can obtain the same result as stated in this theorem, with the only
%difference that the proper control gains are given by $K_{\lambda} =
%B^{-1}(-A-\lambda I)$, and that these systems are diagonalizable, thus
%there is no need for calling Jordan decomposition
%techniques. Regarding \textit{under-actuated systems}, however, there
%is a specific need to call Jordan decomposition techniques, the main
%issue is that for these systems, we can just prove the
%\textit{existence} of some infinitely-many non-repetitive sequence of
%eigenvalues, e.g., $\{\lambda_{i}\}_{i=1}^{\infty}$, for which the
%geometric multiplicity, and thus the structure of matrix $N$ is
%preserved, nonetheless, no explicit formula for them can be expressed.
%\end{remark}
%\margin{I would leave this remark also out of the paper or have an
%  even shorter version where we don't explain that much. In other
%  words I would say that for other underactuated systems, the results
%  can be extended (without going too much into the details, or leaving
%  them for an extended report). Anyway, I would develop this for your
%  thesis for sure, so when you have time, please write it down.}
%  \marginH{I see, I will put it in the future works section.}

%\begin{remark}
%Practical constraints shrink the domain of validity of the Theorem~\ref{main-thm-arbit-jammer}, in the sense
%if  
%\end{remark}

%%%%%%
\section{Simulations}
\label{sec-simulations}
Having established the theoretical results of previous sections, here
we demonstrate the functionality of these results on some
representative academic examples.

\subsection{Example 1: $3 \times 3$ system}
\label{subsec-sim-33-sys}
We consider the following system:
\begin{align*}
  \dot{x} &= \left[ \begin{array}{ccc} 0 & 1 & 0 \\ 0 & 0 & 1 \\ -3 &
      -2 & 3 \end{array} \right] x + \left[ \begin{array}{c} 0 \\ 0
      \\ 1 \end{array} \right] u\,, \\ u &=
  \left[-\left(\begin{array}{c} 3 \\ 3 \end{array}
    \right)\lambda^{3}+3, -\left(\begin{array}{c} 3 \\ 2 \end{array}
    \right)\lambda^{2}+2, -3\lambda-3\right]x\,.
\end{align*}
The state matrix of the closed-loop system is of the following form:
\begin{equation*}
  A+BK_{\lambda} = \left[ \begin{array}{ccc} 0 & 1 & 0 \\ 0 & 0 & 1
      \\ -\lambda^{3} & -3 \lambda^{2} & -3\lambda \end{array}
    \right]\,,
\end{equation*}
where its only eigenvalue is $-\lambda$, which has algebraic and
geometric multiplicity of $3$, and $1$, respectively, referring to
Proposition~\ref{prop-lambda-eigenvalue-prop-prel}. The only linearly
independent eigenvector is given by solving the equation
$(A+BK_{\lambda}+\lambda I)v_{1} = 0$ for $v_{1}$, where we obtain:
\begin{equation*}
  v_{1} =\left(\begin{array}{c} 1 \\ -\lambda
    \\ \lambda^{2} \end{array} \right)\,.
\end{equation*} 
In order to build the matrix $T_{\lambda}$ (and $T_{\lambda}^{-1}$),
we need to generate two other generalized eigenvectors, namely
$v_{2}$, and $v_{3}$. They are, respectively, the solutions to
$(A+BK_{\lambda} + \lambda I)v_{2} =v_{1}$, and $(A+BK_{\lambda} +
\lambda I)v_{3} =v_{2}$ equations. After some algebraic manipulations,
we get the following result:
\begin{equation*}
  v_{2} = \left( \begin{array}{c} \frac{2}{\lambda} \\ -1
    \\ 0 \end{array} \right)\,, \, v_{3} = \left( \begin{array}{c}
    \frac{3}{\lambda^{2}} \\ -\frac{1}{\lambda} \\ 0 \end{array}
  \right)\,.
\end{equation*}
Hence, matrix $T_{\lambda}$ is obtained as: $T_{\lambda} = [v_{1},
  v_{2}, v_{3}]$. Moreover, given the multiplicities of $-\lambda$,
the matrix $N$ is as follows:
\begin{equation*}
  N = \left[ \begin{array}{ccc} 0 & 1 & 0 \\ 0 & 0 & 1 \\ 0 & 0 &
      0 \end{array} \right]\,.
\end{equation*}
%\marginH{I presume, this paragraph has to be also left out, as the
%  remark isn't there, could you share your opinion?}\margin{Yes, I
%  would leave this out}Before presenting the simulation results and in
%order to back up our explanation in
%Remark~\ref{remark-main-thm-diag-dom}, we compute the following:
%\begin{equation*}
%  (T_{\lambda}^{-1})^{T}T_{\lambda}^{-1} = \left[ \begin{array}{ccc}
%      \lambda^{4}+\lambda^{2} & 2\lambda^{3}+3\lambda & \lambda^{2}+2
%      \\ 2\lambda^{3} + 3 \lambda & 4\lambda^{2}+9 &
%      2\lambda+\frac{6}{\lambda} \\ \lambda^{2}+2 & 2\lambda +
%      \frac{6}{\lambda} & \frac{1}{\lambda^{4}} +
%      \frac{4}{\lambda^{2}} + 1 \end{array}\right]\,,
% \end{equation*}
%and we note that this matrix is \textit{not} diagonally dominant for
%all values of $\lambda$. Hence, Ger$\check{\text{s}}$gorin circle
%theorem could not have been used in the theoretical development in
%Section~\ref{sec-stab-analysis}.

In order to perform the simulation, we have to ``tune'' some
parameters related to the jammer and the triggering policy. We have
chosen $\sigma = 0.1$, jammer action-period $T =1 \sec$,
$\supscr{\subscr{T}{on, 1}}{cr} = 0.9 T$, $\supscr{\subscr{T}{off,
    1}}{cr} = 0.1 T$, and $\supscr{\subscr{T}{on, 2}}{cr} = 0.5 T$,
$\supscr{\subscr{T}{off, 2}}{cr} = 0.5 T$. We note that the first
jammer is more malicious than the second one.

We use the procedure explained in Algorithm~\ref{c-lambda-algorithm}, 
to run the simulation. The result is presented in
Figure~\ref{third-order-sys-tau-lambda}.
 In order to get a deeper insight into the ODE introduced in step~\ref{step-alg-c-lambda-ode} of this
algorithm, refer to Corollary~IV.1 in~\cite{PT:07}. 
%The procedure used to run the simulation is roughly as follows:
%\marginH{to put it in a algorithm box, w/ $T_{on}$, $T_{off}$,
%$\sigma$ as inputs, etc. } \begin{enumerate} \item Pick a sequence of
%eigenvalues, e.g., $\{\lambda_{k} = 100+10k\}_{k=1}^{990}$, \item For
%each $\lambda_{k}$, numerically compute the solution to the following
%ODE: \begin{align*} \dot{\phi} = &\|A+BK_{\lambda_{k}}\| +
%(\|A+BK_{\lambda_{k}}\| + \|BK_{\lambda_{k}}\|) \phi +\\ &
%\|BK_{\lambda_{k}}\| \phi^{2}\,, \end{align*} with the initial
%condition $\phi(0) = 0$. Then, $\tau_{\lambda_{k}}$ is obtained such
%that $\phi(\tau_{\lambda_{k}}) = \sigma$. To get a deeper insight,
%refer to Corollary IV.1 in~\cite{PT:07}.  \item Accordingly, compute
%$C(\lambda_{k})$, and draw $C(\lambda_{k}) \, \text{versus} \,
%\lambda_{k}$.  \end{enumerate}

%%%%%%%%% algorithm
%\margin{looks nice. What I would do is remove the reference in step 3,
%  and put it outside the algorithm environment.}
\begin{algorithm}[tb] \small
\caption{$C(\lambda)$-Seeking} \label{c-lambda-algorithm}

\begin{algorithmic}

\REQUIRE Matrices: $A$, $B$, and $N$, Sequence:
$\{\lambda_{k}\}_{k=1}^{N^{\prime}}$, Parameters: $\sigma$,
$\supscr{\subscr{T}{off}}{cr}$, and $T$.

\end{algorithmic}

\begin{algorithmic}[1]

\STATE Given controllable pair $(A,B)$, compute the proper similarity
transformation matrix, and find $(A_{c}, B_{c})$---which are in
controllable canonical form,

\FOR{$k = 1$ to $N^{\prime}$}

\STATE \label{step-alg-c-lambda-ode} {Numerically solve the following ODE, with $\phi(0) = 0$:
	\begin{align*}
	\dot{\phi} = &\|A+BK_{\lambda_{k}}\| + (\|A+BK_{\lambda_{k}}\|
        + \|BK_{\lambda_{k}}\|) \phi +\\ & \|BK_{\lambda_{k}}\|
        \phi^{2}\,,
	\end{align*}}
\STATE{Find $\tau_{\lambda_{k}}$, such that $\phi(\tau_{\lambda_{k}})
  = \sigma$,}

\STATE{Compute $C(\lambda_{k})$, as stated in
  equation~\eqref{eqn-thm-ineq-orig-dynamics-main}}.

\ENDFOR

\ENSURE{Sequences $\{C(\lambda_{k})\}_{k=1}^{N^{\prime}}$ and
  $\{\tau_{\lambda_{k}}\}_{k=1}^{N^{\prime}}$ }.

\end{algorithmic}
\end{algorithm}
%%%%%%%%%

% rc --> vertical
% tc --> horizontal
% tr --> the origin
\begin{figure}
  \centering \psfrag{lambda}[tc][cc]{$\lambda$}
  \psfrag{c-lambda}[rc][cc]{$C(\lambda)$}
  \psfrag{coeff-90-perc-third}[rc][cc]{$\text{Coefficient}\,
    C(\lambda), \text{90\% active jammer}$}
  \psfrag{coeff-50-perc-third}[rc][cc]{$\text{Coefficient}\,
    C(\lambda), \text{50\% active jammer}$} \psfrag{0}[tr][cc]{$0$}
  \psfrag{500}[tc][cc]{$500$} \psfrag{1000}[tc][cc]{$1000$}
  \psfrag{1500}[tc][cc]{$1500$} \psfrag{2000}[tc][cc]{$2000$}
  \psfrag{1}[rc][cc]{$1$} \psfrag{2}[rc][cc]{$2$}
  \psfrag{3}[rc][cc]{$3$} \includegraphics[width=3.5in,
    height=3.0in]{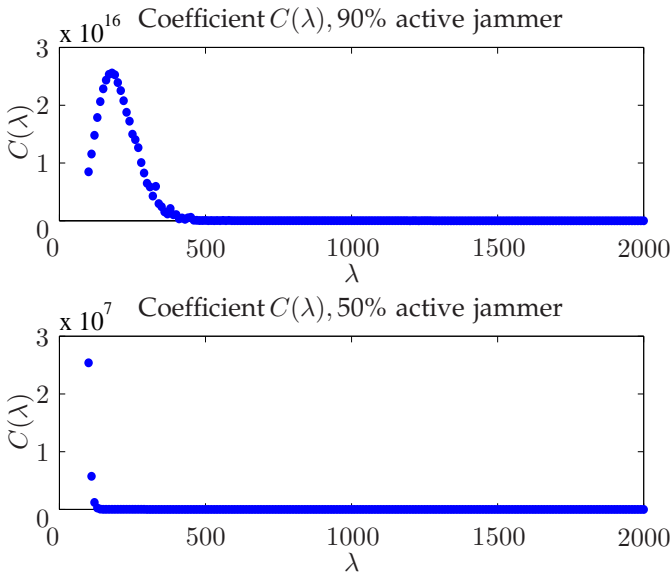}
  \caption{Third-order system: comparing $90\%$ and $50\%$ active jammers}
  \label{third-order-sys-c-lambda}
\end{figure}

Referring to Figure~\ref{third-order-sys-c-lambda}, we can list the
following remarks:
\begin{remark}
  We acknowledge these facts: (i) the $90\%$ active jammer is more
  malicious than the $50\%$ active jammer, and (ii) to maintain the
  asymptotic stability, we should at least guarantee $C(\lambda) <
  1$. Hence, let us define:
\begin{equation*}
  \bar{\lambda} = \min_{1 \leq k \leq N^{\prime}} {\{\lambda_{k} | \forall
    \lambda \geq \lambda_{k}, C(\lambda) < 1\}}\,.
\end{equation*}
Then, we obtain $\bar{\lambda}_{90\%} = 1360$, and
$\bar{\lambda}_{50\%} = 210$. Accordingly, we can induce that in order
to guarantee the asymptotic stability, \textit{larger} poles (in the
absolute sense) are required in the case of $90\%$ active jammer; this
can be interpreted as \text{larger} control effort.
\end{remark}

\begin{remark}
  According to the intricacy of the $C(\lambda)$ equation; finding
  $\lambda$ analytically, for a given value of $C(\lambda)$ is not
  feasible. As an alternative way, one can use our proposed procedure,
  in order to \textit{numerically} achieve this goal. So, e.g., having
  obtained the sequence $\{C(\lambda_{k})\}_{k=1}^{N^{\prime}}$, for a given
  sequence $\{\lambda_{k}\}_{k=1}^{N^{\prime}}$, one can then obtain a
  polynomial or spline approximation for $C(\lambda)$.
\end{remark}

For this system, we have also conducted a study on the evolution of
the parameter $\tau_{\lambda}$. This time, we picked the sequence
$\{\lambda_{k} = 0.01k\}_{k=1}^{1000}$, and for each $\lambda_{k}$, we
run the procedure explained in Algorithm~\ref{c-lambda-algorithm}. The
result is presented in Figure~\ref{third-order-sys-tau-lambda}.

% rc --> vertical
% tc --> horizontal
% tr --> the origin
\begin{figure}
  \centering \psfrag{lambda}[tc][cc]{$\lambda$}
  \psfrag{tau-lambda}[rc][cc]{$\tau_{\lambda}$}
  \psfrag{tau-lambda-evo}[rc][cc]{$\text{Evolution of the parameter}\,
    \tau_{\lambda}$} \psfrag{0}[tr][cc]{$0$} \psfrag{2}[tc][cc]{$2$}
  \psfrag{4}[tc][cc]{$4$} \psfrag{6}[tc][cc]{$6$}
  \psfrag{8}[tc][cc]{$8$} \psfrag{10}[tc][cc]{$10$}
  \psfrag{0.01}[rc][cc]{$0.01$} \psfrag{0.02}[rc][cc]{$0.02$}
  \psfrag{0.03}[rc][cc]{$0.03$} \psfrag{0.04}[rc][cc]{$0.04$}
  \psfrag{0.05}[rc][cc]{$0.05$} \psfrag{0.06}[rc][cc]{$0.06$}
  \psfrag{0.07}[rc][cc]{$0.07$} \psfrag{0.08}[rc][cc]{$0.08$}
  \includegraphics[width=3.5in,
    height=3.0in]{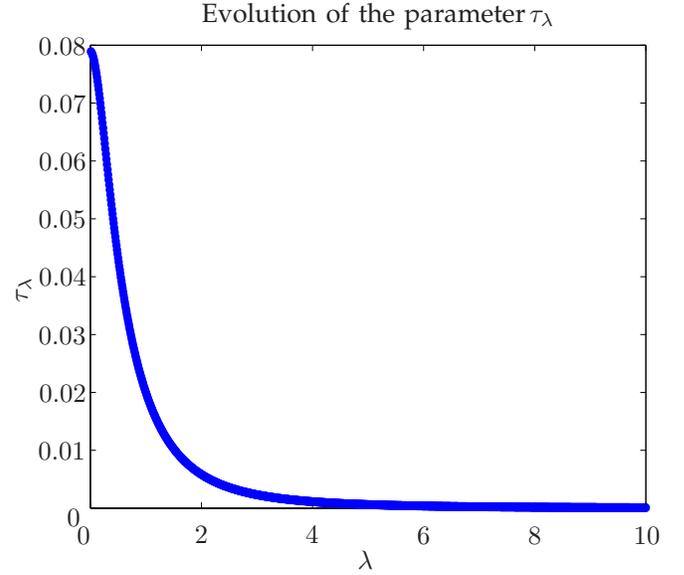}
  \caption{Third-order system: evolution of $\tau_{\lambda}$}
  \label{third-order-sys-tau-lambda}
\end{figure}
%\margin{I would remove this remark environment. }
The Figure~\ref{third-order-sys-tau-lambda} confirms our result in
  Theorem~\ref{thm-jordan-trig-seq-tau-lambda} on the evolution of
  $\tau_{\lambda}$.

\subsection{Example 2: $5 \times 5$ system}
\label{subsec-sim-55-sys}
Our main result in this paper, Theorem~\ref{main-thm-arbit-jammer}, is
stated for the systems of order $4$, or less. Nevertheless, this is
based on the general condition provided by Galois Theory, leaving open
the question of whether it holds for subclasses of systems of higher
order. We have conducted a simulation study on a $5 \times 5$ system
in canonical form, and as it comes later, our result is yet valid.

We consider the following system:
\begin{align*}
  \dot{x} =& \left[ \begin{array}{ccccc} 0 & 1 & 0 & 0 & 0 \\ 0 & 0 &
      1 & 0 & 0 \\ 0 & 0 & 0 & 1 & 0 \\ 0 & 0 & 0 & 0 & 1 \\ -7 & 10 &
      -3 & 4 & -6 \end{array} \right] x + \left[ \begin{array}{c} 0
      \\ 0 \\ 0 \\ 0 \\ 1 \end{array} \right] u\,, \\ u =&
      [-\lambda^{5}+7, -5\lambda^{4}-10, -10\lambda^{3}+3,
        \\ &-10\lambda^{2}-4, -5\lambda+6 ]x\,.
\end{align*}
The state-matrix of the closed-loop system is of the following form:
\begin{equation*}
  A+BK_{\lambda} = \left[ \begin{array}{ccccc} 0 & 1 & 0 & 0 & 0 \\ 0
      & 0 & 1 & 0 & 0 \\ 0 & 0 & 0 & 1 & 0 \\ 0 & 0 & 0 & 0 & 1
      \\ -\lambda^{5} & -5 \lambda^{4} & -10\lambda^{3} &
      -10\lambda^{2} & -5\lambda \end{array} \right]\,,
\end{equation*}
where its only eigenvalue is $-\lambda$, which has algebraic and
geometric multiplicity of $5$, and $1$, respectively -referring to
Proposition~\ref{prop-lambda-eigenvalue-prop-prel}. The only linearly
independent eigenvector is given by solving the equation
$(A+BK_{\lambda}+\lambda I)v_{1} = 0$ for $v_{1}$, where we obtain:
\begin{equation*}
  v_{1} = \left(\begin{array}{c} 1 \\ -\lambda \\ \lambda^{2}
    \\ -\lambda^{3} \\ \lambda^{4} \end{array} \right)\,.
\end{equation*} 
In an analogous way as in Subsection~\ref{subsec-sim-33-sys}, we
compute the generalized eigenvectors:
\begin{align*}
  v_{2} =& \left( \begin{array}{c} \frac{4}{\lambda} \\ -3 \\ 2\lambda
    \\ -\lambda^{2} \\ 0 \end{array} \right)\,, \, v_{3} =
  \left( \begin{array}{c} \frac{10}{\lambda^{2}} \\ -\frac{6}{\lambda}
    \\ 3 \\ -\lambda \\ 0 \end{array} \right)\,, \, v_{4} =
  \left( \begin{array}{c} \frac{20}{\lambda^{3}}
    \\ -\frac{10}{\lambda^{2}} \\ \frac{4}{\lambda} \\ -1
    \\ 0 \end{array} \right)\,, \\ v_{5} =& \left( \begin{array}{c}
    \frac{35}{\lambda^{4}} \\ -\frac{15}{\lambda^{3}}
    \\ \frac{5}{\lambda^{2}} \\ -\frac{1}{\lambda} \\ 0 \end{array}
  \right)\,.
\end{align*}
Hence, matrix $T_{\lambda}$ is obtained as: $T_{\lambda} = [v_{1},
  v_{2}, v_{3}, v_{4}, v_{5}]$. Moreover, given the multiplicities of
$-\lambda$, the matrix $N$ is as follows:
\begin{equation*}
  N = \left[ \begin{array}{ccccc} 0 & 1 & 0 & 0 & 0 \\ 0 & 0 & 1 & 0 &
      0 \\ 0 & 0 & 0 & 0 & 1 \\ 0 & 0 & 0 & 0 & 0 \end{array}
    \right]\,.
\end{equation*}

We ``tune'' the parameters related to the jammer and the triggering
policy: $\sigma = 0.1$, jammer action-period $T =1 \sec$,
$\supscr{\subscr{T}{on, 1}}{cr} = 0.9 T$, $\supscr{\subscr{T}{off,
    1}}{cr} = 0.1 T$, and $\supscr{\subscr{T}{on, 2}}{cr} = 0.5 T$,
$\supscr{\subscr{T}{off, 2}}{cr} = 0.5 T$. We note that the first
jammer is more malicious than the second one. Then, we perform the
simulation running the procedure explained in
Algorithm~\ref{c-lambda-algorithm}, the result is shown in
Figure~\ref{fifth-order-sys-c-lambda}.

% rc --> vertical
% tc --> horizontal
% tr --> the origin
\begin{figure}
\centering \psfrag{lambda}[tc][cc]{$\lambda$}
\psfrag{c-lambda}[rc][cc]{$C(\lambda)$}
\psfrag{coeff-90-perc-third}[rc][cc]{$\text{Coefficient}\, C(\lambda),
  \text{90\% active jammer}$}
\psfrag{coeff-50-perc-third}[rc][cc]{$\text{Coefficient}\, C(\lambda),
  \text{50\% active jammer}$} \psfrag{0}[tr][cc]{$0$}
\psfrag{600}[tc][cc]{$600$} \psfrag{800}[tc][cc]{$800$}
\psfrag{1000}[tc][cc]{$1000$} \psfrag{1200}[tc][cc]{$1200$}
\psfrag{1400}[tc][cc]{$1400$} \psfrag{1600}[tc][cc]{$1600$}
\psfrag{1800}[tc][cc]{$1800$} \psfrag{5}[rc][cc]{$5$}
\psfrag{10}[rc][cc]{$10$} \psfrag{15}[rc][cc]{$15$}
\psfrag{1500}[rc][cc]{$1500$} \psfrag{3000}[rc][cc]{$3000$}
\psfrag{4500}[rc][cc]{$4500$} \includegraphics[width=3.5in,
  height=3.0in]{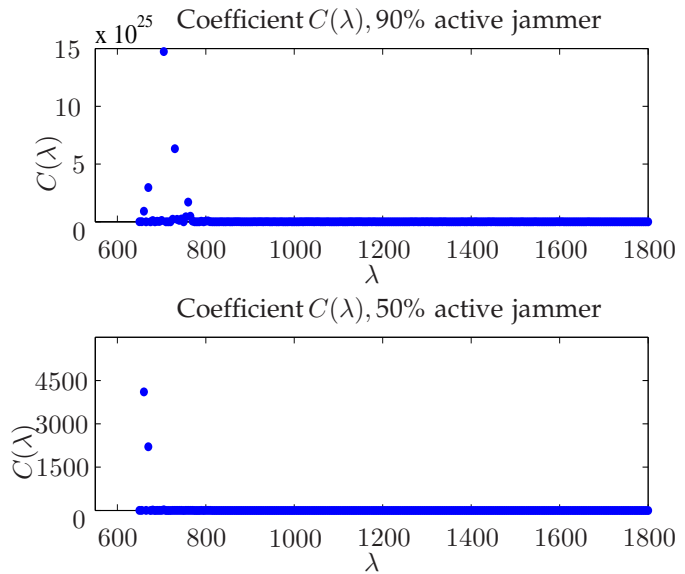}
\caption{Fifth-order system: comparing $90\%$ and $50\%$ active jammers}
\label{fifth-order-sys-c-lambda}
\end{figure}

%\begin{figure}
%\centering
%\includegraphics[width=2.5in]{fifthorder_C_lambda_evo}
%\caption{Fifth-order system: comparing $90\%$ and $50\%$ active jammers}
%\label{fifth-order-sys-c-lambda}
%\end{figure}

Referring to Figure~\ref{fifth-order-sys-c-lambda}, we can list
similar remarks, as in
Subsection~\ref{subsec-sim-33-sys}. Furthermore, as promised before,
in both cases, it holds that $\lim_{\lambda \rightarrow \infty}
{C(\lambda)} = 0$, which could not be theoretically backed up.

\section{Conclusions and Future Work}
\label{sec-conclusions-future-work}
In this paper, we have considered single-input, of order $4$ or less,
continuous LTI systems, under periodic PWM DoS jamming attacks.  We
have proposed a resilient control design law, along with a triggering
time-sequence to update the controller. In the main result, we
demonstrated that this control design and triggering law is capable of
counteracting the effect of any jammer. In other words, we show that
the system is rendered asymptotically stable under our
contributions. The functionality of the theoretical studies has been
demonstrated in the simulation environment; where we have also shown
that the result holds for a system of order $5$, for which our
theoretical result cannot be stretched.

In this work, we have assumed that the jammer signal has been
previously detected and identified. We are currently studying how to
exploit signal processing techniques to partly identify the jammer, that is to
 identify the parameter $T$: the jammer's period
% \margin{the jammer is
%  characterized by two parameters, T and Toff. Should we say partly
%  identify the jammer?}
  . Moreover, as the title stands for, in this
paper we have focused on single-input linear systems. In future work,
we will plan to study nonlinear and multi-input classes of systems.

%Moreover, as it is stated in the problem formulation, we are considering the attacks on the
%control communication channel, as an other future work, we are keen to study the consequences of jamming attacks on 
%the observation channel. A main question would remain as whether or not any correlation exists between control, and observation jamming.

\bibliographystyle{plain}
\bibliography{/Users/hamed/bib/alias,/Users/hamed/bib/Main,/Users/hamed/bib/New,/Users/hamed/bib/FB,/Users/hamed/bib/SMD-add,/Users/hamed/bib/SM,/Users/hamed/bib/Main-sonia,/Users/hamed/bib/New-sonia}
%\bibliography{alias,SMD-add}

\end{document}